\documentclass[12pt, reqno]{amsart}

\usepackage{amsmath,amssymb,amsthm}

\usepackage{amssymb}
\usepackage{amsthm}
\usepackage{amsmath}
\usepackage{color}
\usepackage{ulem}
\usepackage[all]{xy}
\usepackage{graphicx}
\usepackage[margin=30mm]{geometry}

\numberwithin{equation}{section}
\newtheorem{theorem}{Theorem}[section]
\newtheorem{corollary}[theorem]{Corollary}
\newtheorem{lemma}[theorem]{Lemma}
\newtheorem{proposition}[theorem]{Proposition}
\newtheorem{example}[theorem]{Example}
\newtheorem{remark}[theorem]{Remark}
\newtheorem{definition}[theorem]{Definition}

\begin{document}

\title[Intersection of Real Flag Manifolds]{
The intersection of two real flag manifolds in\\ a complex flag manifold, and 
the canonical form of a compact symmetric triad
}
\author{Kurando Baba}
\address{Department of Mathematics,
Faculty of Science and Technology,
Tokyo University of Science,
2641 Yamazaki, Noda, 278-8510, Chiba, Japan,
Research Institute for Science and Technology at Tokyo University of Science,
Division of Joint Research of Geometry and Natural Science}
\email{kurando.baba@rs.tus.ac.jp}

\author{Osamu Ikawa}
\address{
Faculty of Arts and Sciences,
Kyoto Institute of Technology,
Matsugasaki, Sakyoku,
Kyoto 606-8585, Japan
}
\email{ikawa@kit.ac.jp}

\author{Shinji Ohno}
\address{
Department of Mathematics,
College of Humanities and Sciences,
Nihon University,
3-25-40, Sakurajosui, Setagaya-ku,
Tokyo 156-8550, Japan
}
\email{ohno.shinji@nihon-u.ac.jp}




\thanks{The first author was partially supported by JSPS KAKENHI Grant Number JP25K07017.}.

\begin{abstract}
The necessary and sufficient conditions for the intersection of two real flag manifolds in a complex flag manifold to be discrete 
are stated including a previous known result. 
The discrete intersection of real flag manifolds is the orbit of a Weyl group, which is an ‘antipodal set’ of a complex flag manifold. 
The triple of two real flag manifolds and a complex flag manifold as an ambient space is constructed from a compact symmetric triad. 
It is shown that the compact symmetric triad can be taken to be ‘canonical’.
\end{abstract}

\maketitle

\section{Introduction}

A {\it real form} of a K\"ahler manifold is a connected component of the fixed-point set  of an anti-holomorphic involutive 
isometry, which is a totally geodesic Lagrangian submanifold. 
In this paper, a {\it complex flag manifold} is defined by an adjoint orbit $\mathrm{Ad}(G)x_0$ 
of a compact connected simple Lie group $G$, 
where $x_0(\not=0)$ is an element of the Lie algebra $\mathfrak{g}$ of $G$. 
It is known that a complex flag manifold $\mathrm{Ad}(G)x_0$ is compact and simply connected, and a 
$G$-homogeneous K\"ahler manifold with respect to a $G$-invariant K\"ahler metric (\cite{Borel}). 

A pair of complex flag manifold and its real flag manifold is constructed from 
a compact connected simple Lie group $G$ and its involution $\theta$, 
where an involution means an involutive automorphism. 
Indeed, a real flag manifold is an orbit of the identity component of the 
fixed-point subgroup of $\theta$ (see Section 2 for the precise definition). 
A real flag manifold is a real form of a complex flag manifold. 
A triple of a complex flag manifold and two real flag manifolds is constructed from a triple of a compact connected 
simple Lie group $G$ and its two involutions $\theta_1$ and $\theta_2$. 
A triple $(G,\theta_1,\theta_2)$ is called a {\it compact symmetric triad}. 

In general, the intersection of two real flag manifolds is a totally geodesic submanifold in the complex flag manifold 
by \cite[p.~61]{KN}.
In previous study \cite{IIOST}, some necessary and sufficient conditions for the intersection of two real flag manifolds to be discrete 
were investigated (see (2), (3) of Theorem~\ref{thm:int}). 
In particular, any discrete intersection is an antipodal set of the complex flag manifold
in the sense of \cite{IIOST}.
This article adds a new necessary and sufficient condition which are easily verified (see (4) of Theorem~\ref{thm:int}).
We show that a discrete intersection of two real flag manifolds in a complex flag manifold is an orbit of a Weyl group (Theorem~\ref{thm:int}). 
Consequently, we find that the cardinality of the discrete intersection is at least some constant greater than or equal to two, which is determined by the two involutions (Corollary~\ref{cor:discreteinteraection0}).
This estimate is a refinement of the fact that the intersection is non-empty (see \cite[Lemma~4.1]{IIOST}).
Moreover, we show that this estimate is sharp (see the argument following Corollary~\ref{cor:discreteinteraection0}).

Our study of intersections of real flag manifolds is motivated by the study
of antipodal sets introduced by Chen-Nagano \cite{CN}, by the computation of
Lagrangian Floer homology, and by the generalized Arnold-Givental
conjecture (\cite{IST}, \cite{Oh}).
With these motivations, the present paper extends some previous results of
\cite{ITT2014},  \cite{ITT2015} as well as some results of \cite{IIOST}.
In fact, in \cite{ITT2014}, the second author, Tanaka and Tasaki studied a necessary and sufficient condition for the intersection of two real flag manifolds which are congruent to real Grassmann manifolds in a complex Grassmann manifold to be discrete, and showed that the discrete intersection is an antipodal set of the complex Grassmann manifold, and it is an orbit of a Weyl group.
Subsequently, in \cite{ITT2015},  they dealt with the intersection of two real flag manifolds which
are not necessarily congruent with each other in a Hermitian symmetric space of compact type.
We refer to \cite{IST} for the relation of the study of intersections of real flag manifolds
to the computation of Lagrangian Floer homology and to the generalized Arnold-Givental conjecture.
The study in \cite{IST} on intersections of two real flag manifolds in irreducible Hermitian symmetric spaces was extended
from irreducible Hermitian symmetric spaces to complex flag manifolds in \cite{IIOST}.
While \cite{IIOST} required the commutativity of the compact symmetric triad realizing
the two real forms in a complex flag manifold, in the present paper
we extend the results of \cite{IIOST} without assuming the commutativity
by applying the structure theory of compact symmetric triads mentioned below.%

Matsuki (\cite{Matsuki}, \cite{Matsuki2}) introduced a nontrivial equivalence relation on the set of compact symmetric triads. 
The first and second authors (\cite{BI}) showed that there exists a representative which has good properties in each equivalence class, which we called a {\it canonical form}.  
In Section~\ref{sec:main2} we formulate the problem about the intersection of two real flag manifolds 
in the viewpoint of canonical forms (Theorem~\ref{thm:can2}), 
and show some properties under the viewpoint (Theorems~\ref{thm:discrete}$\sim$\ref{thm:-1inv}).  
By using canonical form theory, we can and do study the intersection of two real flag manifolds without the commutativity condition. 
In addition, the third author's results in \cite{Ohno} is a useful tool to analyze 
the intersection.
For instance, by using the results in \cite{Ohno}, we can calculate the dimension of the intersection whether the intersection is discrete or not. 
In particular, we can and do provide the necessary and sufficient condition for the intersection to be discrete.
In this way, the formulation provides a new perspective for the study of the intersection theory.
In addition, we can use double Satake diagrams defined by the first and second authors in \cite{BI}. 
The triple of a complex flag manifold and two real flag manifolds in it is completely determined 
by the data of the double Satake diagram (see Section~\ref{sec:main2}). 
Hence there is a possibility that we can calculate various quantities related to 
the triple of complex flag manifold and its two real flag manifolds using 
double Satake diagrams even if it is difficult to calculate it directly. 
In Section~\ref{sec:main2} 
we discuss the classification principle for complex flag manifolds and pairs of real flag manifolds in them.

\section{The intersection of two real flag manifolds in a complex flag manifold}

Throughout this paper, let $G$ denote a compact connected simple Lie group, 
and $\mathfrak{g}$ denote the Lie algebra of $G$.  
Take a bi-invariant Riemannian metric $\langle\;,\;\rangle$ on $G$. 
In this paper, for a set $\mathcal{S}$ and a map $\phi:\mathcal{S}\to\mathcal{S}$,
we set $F(\phi,\mathcal{S})=\{x\in\mathcal{S}\mid \phi(x)=x\}$.

For an involution $\theta$ on $G$,
define the subalgebra $\mathfrak{k}$ of $\mathfrak{g}$, and the subspace $\mathfrak{p}$ of $\mathfrak{g}$ by 
$\mathfrak{k}=F(\theta,\mathfrak{g})$
and $\mathfrak{p}=F(-\theta,\mathfrak{g})$,
respectively,
where $\theta$ denotes an involution on $\mathfrak{g}$
which is induced from $\theta$ on $G$.
Then $\mathfrak{k}$ coincides with the Lie algebra of the closed subgroup $F(\theta,G)$ of $G$, and $\mathfrak{g}=\mathfrak{k}\oplus\mathfrak{p}$ is a 
direct sum decomposition of $\mathfrak{g}$, which we call the canonical decomposition of 
$\mathfrak{g}$ with respect to $\theta$. 
For $x_0\in \mathfrak{p}-\{0\}$, set $M=\mathrm{Ad}(G)x_0$. 
It is known by Borel (\cite{Borel}) that $M$ is compact and simply connected, and has a structure of $G$-homogeneous K\"{a}hler manifold, 
which we call the {\it complex flag manifold}.  
Set $L=\mathrm{Ad}(F(\theta,G))x_0$, where $\mathrm{Ad}$ is the adjoint representation of $G$ restricted to $F(\theta,G)$. 
Then it is known by \cite[Prop. 2.1]{HKT} that $L$ is connected and $L=\mathrm{Ad}(F(\theta,G)_0)x_{0}$, where 
$F(\theta,G)_0$ is the identity component of $F(\theta,G)$. 
Let $K$ be a closed subgroup of $G$ which satisfies 
$F(\theta,G)_0\subset K\subset F(\theta,G)$, that is, $(G,K)$ is a 
Riemannian symmetric pair. 
Based on the above, we have $L=\mathrm{Ad}(K)x_0$.
Furthermore $L$ is a real form of $M$ (see \cite{IIOST}, \cite{Ohnita}).
We call $L$ the {\it real flag manifold} of $M$.

\begin{proposition}
The flag manifolds $M$ and $L$ are full in $\mathfrak{g}$ and $\mathfrak{p}$, respectively. 
That is, $\mathrm{span}(M)=\mathfrak{g}$ and $\mathrm{span}(L)=\mathfrak{p}$ hold, where 
$\mathrm{span}(M)$ and $\mathrm{span}(L)$ denote the linear span of $M$ and $L$, respectively.
\end{proposition}
\begin{proof}
We show the first part. 
Since $M$ is an $\mathrm{Ad}(G)$-orbit, $\mathrm{span}(M)$ is an $\mathrm{Ad}(G)$-invariant subspace of $\mathfrak{g}$. 
Hence, $\mathrm{span}(M)$ is an ideal of $\mathfrak{g}$. 
Since $M$ contains a non-zero element $x_0$, we have $\mathrm{span}(M)\not=\{0\}$. 
Since $\mathfrak{g}$ is simple, we have $\mathrm{span}(M)=\mathfrak{g}$. 

Next, we show the latter part. 
Since $L$ is an $\mathrm{Ad}(K)$-orbit, $\mathrm{span}(L)$ is an $\mathrm{Ad}(K)$-invariant subspace of $\mathfrak{p}$. 
Since $L$ contains a non-zero element $x_0$, we have $\mathrm{span}(L)\not=\{0\}$. 
Since $G$ is simple, the compact symmetric pair $(G,K)$ is irreducible (\cite[p.~379,~Theorem~5.3]{Helgason}). 
Thus, we get $\mathrm{span}(L)=\mathfrak{p}$.
\end{proof}

We study the intersection of two real flag manifolds in a complex flag manifold. 
In order to do this, let $\theta_1$ and $\theta_2$ be two involutions on $G$, and 
$\mathfrak{g}=\mathfrak{k}_i\oplus\mathfrak{p}_i$ be the canonical decomposition of $\mathfrak{g}$ with respect to $\theta_i$ for 
$i=1,2$.
Let $K_1$ and $K_2$ be two closed subgroups of $G$ which satisfy $F(\theta_i,G)_0\subset K_i\subset F(\theta_i,G)$ for $i=1,2$.
Take $x_i\in \mathfrak{p}_i-\{0\}$. 
We assume that  there exists $g\in G$ such that $x_2=\mathrm{Ad}(g)x_1$. 
Set $M=\mathrm{Ad}(G)x_1=\mathrm{Ad}(G)x_2$ and $L_i=\mathrm{Ad}(K_i)x_i$. 
Then it is known by \cite{IIOST} that $L_1\cap L_2\not=\emptyset$. 
Hence in order to study $L_1\cap L_2$, we can assume $\mathfrak{p}_1\cap \mathfrak{p}_2\not=\{0\}$ and set the following. 
For $x_0\in \mathfrak{p}_1\cap \mathfrak{p}_2-\{0\}$, set
\[
L_i:=\mathrm{Ad}(K_i)x_0\subset M:=\mathrm{Ad}(G)x_0.
\]
It is known by a theorem of S.~Kobayashi (\cite[p.~61]{KN}) that $L_1\cap L_2$ is a totally geodesic submanifold of $M$, and by 
\cite{IIOST} that 
\begin{equation}
\label{eqn:intersection}
L_1=M\cap {\mathfrak p}_1,\quad L_2=M\cap {\mathfrak p}_2,\quad 
L_1\cap L_2=M\cap ({\mathfrak p}_1\cap {\mathfrak p}_2). 
\end{equation}
Take a maximal abelian subspace $\mathfrak{a}$ of ${\mathfrak p}_1\cap {\mathfrak p}_2$ containing $x_0$. 
Then we have $\dim \mathfrak{a}\geq 1$. 
Let ${\mathfrak g}^{\mathbb C}$ denote the complexification of $\mathfrak{g}$. 
For $\alpha\in\mathfrak{a}$, we define the subspace $\mathfrak{g}(\mathfrak{a},\alpha)$ of ${\mathfrak g}^{\mathbb C}$ by
\begin{equation}
\label{eqn:compg}
\mathfrak{g}(\mathfrak{a},\alpha)=
\{X\in {\mathfrak g}^{\mathbb C}\mid [H,X]=\sqrt{-1}\langle\alpha,H\rangle X\quad (H\in\mathfrak{a})\}.
\end{equation}
Then we have 
$[\mathfrak{a},\mathfrak{g}(\mathfrak{a},\alpha)]
\subset \mathfrak{g}(\mathfrak{a},\alpha)$. 
Let $\bar{ }$ denote the complex conjugation of $\mathfrak{g}^{\mathbb C}$ with respect to $\mathfrak{g}$. 
Then $\overline{\mathfrak{g}(\mathfrak{a},\alpha)}=\mathfrak{g}(\mathfrak{a},-\alpha)$. 
The space $\mathfrak{g}(\mathfrak{a},\alpha)$ is invariant under the composition $\theta_1\theta_2$. 
We define the subset $\tilde{\Sigma}$ of $\mathfrak{a}$ by 
\[
\tilde{\Sigma}:=\{\alpha\in\mathfrak{a}-\{0\}\mid \mathfrak{g}(\mathfrak{a},\alpha)\not=\{0\}\}.
\]

Let $\#(X)$ denote the cardinality of a finite set $X$.
Then the following lemma is known.

\begin{lemma}\label{lem:tildeSigma}
\begin{enumerate}
\item[(1)] $\tilde{\Sigma}$ is an irreducible root system, which spans $\mathfrak{a}$.
\item[(2)] Let $W(\tilde{\Sigma})$ denote the Weyl group of $\tilde{\Sigma}$, and
 $\mathrm{span}(W(\tilde{\Sigma})x_0)$ the linear span of the orbit 
$W(\tilde{\Sigma})x_0$ for $x_0\in \mathfrak{a}-\{0\}$. 
Then we have $\mathrm{span}(W(\tilde{\Sigma})x_0)=\mathfrak{a}$
and $\#(W(\tilde{\Sigma})x_0)\geq \dim\mathfrak{a}+1$.
\end{enumerate}
\end{lemma}
\begin{proof}
We give a brief proof for the completeness.

(1) T.~Matsuki \cite[p.~60, Proposition 1]{Matsuki} proved that $\tilde{\Sigma}$ satisfies the axiom of root system. 
The second author
proved that $\tilde{\Sigma}$ spans $\mathfrak{a}$ in \cite[Lemma~4.12]{ikawa2011}. 
He also proved that the root system $\tilde{\Sigma}$ is irreducible in \cite[Lemma~4.34]{ikawa2011}. 
Although Lemma~4.34 in \cite{ikawa2011} is placed in the subsection 
entitled ``When $G$ is simple and $\theta_1\theta_2=\theta_2\theta_1$ ",  
the lemma and its proof still hold without the assumption $\theta_1\theta_2=\theta_2\theta_1$.

(2) By (1) the natural action of $W(\tilde{\Sigma})$ on $\mathfrak{a}$ is irreducible (\cite[p.~53, Lemma B]{Humphreys}). 
Since $\mathrm{span}(W(\tilde{\Sigma})x_0)(\not=\{0\})$ is a $W(\tilde{\Sigma})$-invariant subspace of $\mathfrak{a}$, 
we have $\mathrm{span}(W(\tilde{\Sigma})x_0)=\mathfrak{a}$. 
Set $r=\dim\mathfrak{a}$. 
It is known that $\#(W(\tilde{\Sigma}))>r!$. 
Since $\mathrm{span}(W(\tilde{\Sigma})x_0)=\mathfrak{a}$, we have $\#(W(\tilde{\Sigma})x_0)\geq r$.
If  $\#(W(\tilde{\Sigma})x_0)=r$, then $W(\tilde{\Sigma})x_0=\{v_1,\ldots, v_r\}$ forms a basis of $\mathfrak{a}$. 
Since each element of $W(\tilde{\Sigma})$ is a permutation of $\{v_1,\ldots,v_r\}$, we have 
$\#(W(\tilde{\Sigma}))\leq r!$, which is a contradiction. 
Thus we get $\#(W(\tilde{\Sigma})x_0))\geq r+1$. 
\end{proof}

Base on the fact above, we call $\tilde{\Sigma}$ the {\it restricted root system} of $\mathfrak{g}$ with respect to $\mathfrak{a}$. 
The inequality $\#(W(\tilde{\Sigma})x_0)\geq \dim\mathfrak{a}+1$ in Lemma~\ref{lem:tildeSigma} is sharp in the sense 
that there is an  example where the equality holds. 

\begin{example}\label{lem:Ar}
In the case when 
\[
\tilde{\Sigma}=A_r=\{e_i-e_j\mid 1\leq i, j\leq r+1, i\not=j\},
\]
where $\mathfrak{a}=\{\sum_{i=1}^{r+1}x_ie_i\mid 
\sum_{i=1}^{r+1}x_i=0\}$ and $\langle e_i, e_j\rangle =\delta_{ij}$.  
Set $x_0=e_1+\cdots +e_r-re_{r+1}\in\mathfrak{a}-\{0\}$. 
Since 
\[
W(\tilde{\Sigma})x_0=
\{e_1+\cdots+\hat{e_i}+\cdots +e_{r+1}-re_i\mid 1\leq i\leq r+1\}
\]
we have $\#(W(\tilde{\Sigma})x_0))=r+1$.\hfill$\Box$

\end{example}

Take a maximal abelian subalgebra $\mathfrak{t}$ of $\mathfrak{g}$ containing $\mathfrak{a}$. 
Let $W(\mathfrak{t})$ denote the Weyl group of $\mathfrak{g}$ with respect to $\mathfrak{t}$. 
Then we have the following lemma. 

\begin{lemma}\label{lem:Mcapa}
$M\cap\mathfrak{a}=W(\mathfrak{t})x_0\cap \mathfrak{a}$.
\end{lemma}
\begin{proof}
By using \cite[p.~285, Proposition 2.2]{Helgason}, we have
\[
M\cap\mathfrak{a}
=M\cap (\mathfrak{t}\cap\mathfrak{a})=(M\cap\mathfrak{t})\cap\mathfrak{a}=W(\mathfrak{t})x_0\cap \mathfrak{a}.
\]
\end{proof}

A subset ${\mathcal A}$ of $M(\subset\mathfrak{g})$ is called an {\it antipodal set} of $M$, if 
$\mathcal{A}$ is an abelian subset, that is, $[{\mathcal A},{\mathcal A}]=\{0\}$ holds. 
The definition was given in \cite{IIOST}. 
It is known by \cite[Theorem~2.2]{IIOST} that an antipodal set of $M$ is finite. 
An irreducible Hermitian symmetric space of compact type is a kind of complex flag manifold. 
When $M$ is an irreducible Hermitian symmetric space of compact type, the notion of an antipodal set 
coincides with that in the sense of Chen-Nagano (\cite{CN}), which was defined by using geodesic symmetries. 
This fact was proved by Tanaka-Tasaki \cite[Theorem~3.1]{TT}.

Now we are in the position to state a main result. 

\begin{theorem}\label{thm:int}
Let $L_1$ and $L_2$ be two real flag manifolds in a complex flag manifold $M$. 
Then the following four conditions are equivalent. 
\footnote{
If $\mathrm{span}(L_1\cap L_2)={\mathfrak p}_1\cap {\mathfrak p}_2$, then 
the conditions (3) and (4) are equivalent. 
But we do not know whether $\mathrm{span}(L_1\cap L_2)={\mathfrak p}_1\cap {\mathfrak p}_2$ holds or not. 

Since $\mathrm{Ad}(K_1\cap K_2)x_0$ is a subset of $L_1\cap L_2$, 
if $\mathrm{span}\{\mathrm{Ad}(K_1\cap K_2)x_0\}={\mathfrak p}_1\cap {\mathfrak p}_2$ then 
$\mathrm{span}(L_1\cap L_2)={\mathfrak p}_1\cap {\mathfrak p}_2$. 
The condition $\mathrm{span}\{\mathrm{Ad}(K_1\cap K_2)x_0\}={\mathfrak p}_1\cap {\mathfrak p}_2$ holds 
for any $x_0\in {\mathfrak p}_1\cap {\mathfrak p}_2-\{0\}$ is equivalent to the condition 
the compact symmetric pair $(G_{12},K_1\cap K_2)$ is irreducible.
But this condition does not hold in general. 
}
\begin{enumerate}
\item[(1)] The intersection $L_1\cap L_2$ of $L_1$ and $L_2$ is discrete. 
\item[(2)] $L_1\cap L_2$ is an antipodal set of $M$, that is, $[L_1\cap L_2,L_1\cap L_2]=\{0\}$. 
\item[(3)] $[{\mathfrak p}_1\cap {\mathfrak p}_2,L_1\cap L_2]=\{0\}$.
\item[(4)] ${\mathfrak p}_1\cap {\mathfrak p}_2$ is an abelian subspace. 
\end{enumerate}
In these cases, the discrete intersection $L_1\cap L_2$ and its linear span are given by 
\begin{align}
& \label{eqn:intersect}
L_1\cap L_2=M\cap \mathfrak{a}=W(\mathfrak{t})x_0\cap\mathfrak{a}
=W(\tilde{\Sigma})x_0,\\
& \label{eqn:lspan}
\mathrm{span}(L_1\cap L_2)=\mathfrak{a}={\mathfrak p}_1\cap {\mathfrak p}_2.
\end{align}
\end{theorem}
\begin{remark}\rm 
As we mention below, it is clear that $(4)\Rightarrow (3)\Rightarrow (2)\Rightarrow (1)$. 
It was proved in \cite{IIOST} that (1) implies (3). 
It was also proved in \cite{IIOST} that (\ref{eqn:intersect}) holds when the commutativity condition 
$\theta_1\theta_2=\theta_2\theta_1$ holds. 

In this paper, we will show that (1) implies (4).
We will also prove that (\ref{eqn:intersect}) holds without the commutativity condition. 
\end{remark}

\begin{proof}[Outline of the proof]
It is clear that (4)$\Rightarrow$(3)$\Rightarrow$(2) since $L_i$ is a subset of ${\mathfrak p}_i$. 
Since an antipodal set of $M$ is finite, we get (2)$\Rightarrow$(1). 
When the condition (4) holds, we have by (\ref{eqn:intersection}) and Lemma~\ref{lem:Mcapa}
\[
L_1\cap L_2=M\cap\mathfrak{a}=W(\mathfrak{t})x_0\cap \mathfrak{a}.
\]
By using Lemma~\ref{lem:tildeSigma}, (2), the equation (\ref{eqn:lspan}) is obtained from (\ref{eqn:intersect}). 
Hence it is sufficient to prove the following: 

$\bullet$ The condition (1) implies the condition (4). 

$\bullet$ $W(\mathfrak{t})x_0\cap \mathfrak{a}=W(\tilde{\Sigma})x_0$ holds for any 
$x_0\in\mathfrak{a}$ regardless whether the intersection $L_1\cap L_2$ is discrete or not. 

We will show these in the next section. 
\end{proof}

\begin{corollary}\label{cor:discreteinteraection0}
Set $r=\dim\mathfrak{a}(\geq 1)$. 
Then the cardinality $\#(L_1\cap L_2)$ of the discrete intersection $L_1\cap L_2$ is greater than or equal to $r+1$.
In particular, $\#(W(\tilde{\Sigma})x_0))\geq 2$
\end{corollary}
\begin{proof}
We have 
$\#(L_1\cap L_2)=\#(W(\tilde{\Sigma})x_0)\geq r+1\geq 2$ by Lemma ~\ref{lem:tildeSigma}, (2) and (\ref{eqn:intersect}).
\end{proof}

The inequality $\#(L_1\cap L_2)\geq r+1$ in Corollary~\ref{cor:discreteinteraection0} is sharp in the sense 
that there are some  examples where the equality holds. 
Since there are some compact symmetric triads $(G,\theta_1,\theta_2)$ such that $\tilde{\Sigma}=A_r$, we have 
$\#(L_1\cap L_2)=r+1$ in this case if we take $x_0$ as in Example~\ref{lem:Ar}.

\begin{example}\rm
In our setting, there are several examples $(G,\theta_1,\theta_2)$ which satisfy $\tilde{\Sigma}=A_r\ (r \geq 2)$. 
The simplest examples are in the case when two real flag manifolds $L_1$ and $L_2$ are congruent with each other. 
For a compact symmetric pair $(G, K)$, let $\theta$ denote the involution of $G$ that determines $K$. 
We assume that its restricted root system $\tilde{\Sigma}$ is of type $A_r$. 
If $G$ is simple,  $(G,K) = (SU(r+1), SO(r+1)), (SU(2(r+1)), Sp(r+1))$, and $(E_6, F_4)$ are such examples. 
Here $r=2$ in the case when $(G,K) =(E_6, F_4)$ (see \cite{Helgason}).
In each example, the pair of orbits $(M, L) = (\mathrm{Ad}(G)x_0, \mathrm{Ad}(K)x_0)$ for $G$ and $K$ passing through $x_0$ as in Example~\ref{lem:Ar},
is a pair of a Hermitian symmetric space and a symmetric R-space, and specifically takes the following forms:
\[
(M, L)=(\mathbb{C}P^{r}, \mathbb{R}P^{r}), (Gr_{2}(\mathbb{C}^{2(r+1)}), \mathbb{H}P^{r}), 
(E_6/(Spin(10)\times_{\mathbb{Z}_{4}} U(1)),\mathbb{O}P^{2}).
\]
In fact,  when $(G,K) = (SU(r+1), SO(r+1))$, we can take 
$$x_{0}=\sqrt{-1}diag(1,\dotsc, 1,-r).$$
Then we have $\mathrm{Ad}(G)x_0\cong SU(r+1)/S(U(r)\times U(1)) =\mathbb{C}P^r$ and 
$\mathrm{Ad}(K)x_0\cong  SO(r+1)/S(O(r)\times O(1))=\mathbb{R}P^r$.

When $(G,K)=(SU(2(r+1)),Sp(r+1))$, we can take 
$$x_{0}=\sqrt{-1}diag(1,\dotsc,1,-r, 1,\dotsc,1,-r).$$
Then we have $\mathrm{Ad}(G)x_0\cong SU(2(r+1))/S(U(2r)\times U(2)) =Gr_2(\mathbb{C}^{2(r+1)})$
and 
$\mathrm{Ad}(K)x_0 \cong Sp(r+1)/(Sp(r)\times Sp(1)) = \mathbb{H}P^r.$
For the exceptional pair $(E_6, F_4)$,   there is  a unique pair of a Hermitian symmetric space $M$ and a  symmetric R-space $L$ as orbits.
Set $\theta_1=\theta$. 
For $a\in A=\exp\mathfrak{a}$, set $\theta_2=\tau_a\theta \tau_a^{-1}$, 
where $\tau_a$ is an inner automorphism of $G$ defined by $a\in A$. 
If we take $a\in A$ such that $L_1\cap L_2$ is discrete, we have $\#(L_1\cap L_2)=r+1$.

Furthermore, if we consider the case when $(G, K_1, K_2) = (SU(2(r+1)), Sp(r+1), SO(2(r+1)))$ and $(E_6, F_4, Sp(4))$, 
then the constructed two real flag manifolds that are not congruent with each other.
Here $\tilde{\Sigma}$ is of type $A_2$ in the case when $(G,K_1,K_2)=(E_6, F_4, Sp(4))$ (see \cite{Matsuki2}).
In this case, the triple of orbits $(M, L_1, L_2)$ of $G$ and $K_1, K_2$ passing through $x_0$ is as follows. 
\begin{align*}
(M, L_{1}, L_{2})=&(Gr_{2}(\mathbb{C}^{2(r+1)}), \mathbb{H}P^{r}, Gr_{2}(\mathbb{R}^{2(r+1)}) ), \\
&(E_6/(Spin(10)\times_{\mathbb{Z}_{4}} U(1)), \mathbb{O}P^{2}, Gr_{2}(\mathbb{H}^{4})).
\end{align*}
In fact, when $(G, K_1, K_2) = (SU(2(r+1)), Sp(r+1), SO(2(r+1)))$, 
we can take 
$$x_{0}=\sqrt{-1}diag(1,\dotsc,1,-r, 1,\dotsc,1,-r).$$
Then we have $\mathrm{Ad}(G)x_0\cong SU(2(r+1))/S(U(2r)\times U(2)) =Gr_2(\mathbb{C}^{2(r+1)})$,
$\mathrm{Ad}(K_1)x_0 \cong Sp(r+1)/(Sp(r)\times Sp(1)) = \mathbb{H}P^r$ and
$\mathrm{Ad}(K_2)x_0 \cong SO(2(r+1))/S(O(2r)\times O(2)) =Gr_{2}(\mathbb{R}^{2(r+1)}).$
When $(G, K_1, K_2) = (E_6, F_4, Sp(4))$, we also have $(M, L_{1}, L_{2})=$\linebreak$(E_6/(Spin(10)\times_{\mathbb{Z}_{4}} U(1)), \mathbb{O}P^{2}, Gr_{2}(\mathbb{H}^{4}))$
from the classification.
For $a\in A=\exp\mathfrak{a}$, set $\theta_2'=\tau_a\theta_2 \tau_a^{-1}$. 
If we take $a\in A$ such that $L_1\cap L_2'$ is discrete, we have $\#(L_1\cap L_2')=r+1$. \hfill$\Box$
\end{example}

\section{Proof of Theorem~\ref{thm:int}}

\subsection{Outline of the proof}

In order to prove Theorem~\ref{thm:int}, it is sufficient to prove the following lemma. 

\begin{lemma}\label{lem:intersect}
The following holds. 
\begin{enumerate}
\item[(1)] 
If ${\mathfrak p}_1\cap {\mathfrak p}_2$ is not abelian, then 
$L_1\cap L_2$ is not discrete. 
\item[(2)] $W(\mathfrak{t})x_0\cap \mathfrak{a}=W(\tilde{\Sigma})x_0$ holds for any $x_0\in\mathfrak{a}$. 
\end{enumerate}
\end{lemma}

After some preparations,  we will show the lemma above.
The proof is given in Subsection~\ref{subsec:proof}.  

A triple $(G,\theta_1,\theta_2)$ of a compact connected simple Lie group $G$ and its two involutions $\theta_1$ and $\theta_2$ 
is called a {\it compact symmetric triad}. 
T.~Matsuki (\cite{Matsuki2}) defined an equivalence relation on the set $\{(G,\theta_1,\theta_2)\}$ of compact symmetric 
triads as follows: 

\begin{definition}\rm \label{dfn:Matsuki}
Two compact symmetric triads $(G,\theta_1,\theta_2)$ and $(G,\theta_1',\theta_2')$ are 
{\it equivalent}, if there exist an automorphism $\varphi \in \mathrm{Aut}(G)$ and an inner automorphism $\tau$
such that 
\begin{equation}
\label{eqn:equiv}
\theta_1'=\tau\varphi\theta_1\varphi^{-1}\tau^{-1},\quad 
\theta_2'=\varphi\theta_2 \varphi^{-1}.
\end{equation}
In this case, we write $(G,\theta_1,\theta_2)\sim (G,\theta_1',\theta_2')$. 
\end{definition}

Under the condition $(G,\theta_1,\theta_2)\sim (G,\theta_1',\theta_2')$,  
we use a primed symbol to denote the object associated to $(G,\theta_1',\theta_2')$ 
corresponding to an object associated to $(G,\theta_1,\theta_2)$. 
For instance, ${\mathfrak p}_i'=F(-\theta_i',\mathfrak{g})$ since we put ${\mathfrak p}_i=F(-\theta_i,\mathfrak{g})$. 
Then we have
\begin{align}
& F(\theta_1',G)=\tau\varphi(F(\theta_1,G)),\quad 
{\mathfrak p}_1'=\tau\varphi({\mathfrak p}_1),\label{eqn:basic0}\\
& F(\theta_2',G)=\varphi (F(\theta_2,G)),\quad 
{\mathfrak p}_2'=\varphi ({\mathfrak p}_2). \label{eqn:basic}
\end{align}

We assume that the complex flag manifold $M = \mathrm{Ad}(G) x_0$ and two real flag manifolds 
$L_i := \mathrm{Ad}(K_i) x_0$ are constructed from the compact symmetric triad $(G, \theta_1, \theta_2)$. 
Under the condition (\ref{eqn:equiv}), we define a complex flag manifold $M'$ by
\[
M'=\mathrm{Ad}(G)\varphi(x_0), 
\]
which is diffeomorphic to $M$. 
Since $\tau$ is an inner automorphism of $G$, 
we have $M'=\mathrm{Ad}(G)\tau \varphi(x_0)$, where we use the same symbol $\tau$ to denote 
the differential of $\tau$. 
Define two real flag manifolds $L_1'$ and $L_2'$ of $M'$ by 
\begin{align}
& L_1'=\mathrm{Ad}(F(\theta_1',G))\tau\varphi(x_0)=\tau\varphi (L_1),\label{eqn:L1}\\
& L_2'=\mathrm{Ad}(F(\theta_2',G))\varphi(x_0)=\varphi(L_2).\label{eqn:L2}
\end{align}
where the differential map of $\varphi$ is also denoted by the same symbol $\varphi$.
The following lemma follows from (\ref{eqn:basic0})$\sim$ (\ref{eqn:L2}). 

\begin{lemma}\label{lem:equiv} 
We retain the notations above. 
\begin{enumerate}
\item[(1)] When $\tau$ is identity in (\ref{eqn:equiv}), then we have
\[\hspace{15mm}
L_i'=\varphi(L_i),\quad 
L_1'\cap L_2'=\varphi(L_1\cap L_2), \quad 
{\mathfrak p}_1'\cap {\mathfrak p}_2'=\varphi({\mathfrak p}_1\cap {\mathfrak p}_2).
\]
The condition $L_1\cap L_2=W(\tilde{\Sigma})x_0$ holds if and only if 
$L_1'\cap L_2'=W(\varphi(\tilde{\Sigma}))\varphi(x_0)$ holds, 
where ${\mathfrak a}':=\varphi(\mathfrak{a})$ is a maximal abelian subspace of ${\mathfrak p}_1'\cap {\mathfrak p}_2'$, and 
$\varphi(\tilde{\Sigma})$ is a restricted root system of $\mathfrak{g}$ with respect to 
$\mathfrak{a}'$. 
\item[(2)] When $\varphi=1$ and $\tau=\tau_{k_2ak_1}\;(k_i\in K_i, a\in A:=\exp\mathfrak{a})$ in (\ref{eqn:equiv}), we have 
\begin{align*}
& L_1'=\mathrm{Ad}(k_2)\mathrm{Ad}(a)L_1,\quad L_2'=L_2,\\ 
& L_1'\cap L_2'=\mathrm{Ad}(k_2)(\mathrm{Ad}(a) L_1\cap L_2),\\
& \mathfrak{p}_1'\cap \mathfrak{p}_2'
=\mathrm{Ad}(k_2)(\mathrm{Ad}(a){\mathfrak p}_1\cap {\mathfrak p}_2),
\end{align*}
where $\mathfrak{a}$ is also a maximal abelian subspace of $\mathrm{Ad}(a) \mathfrak{p}_1 \cap \mathfrak{p}_2$.
The condition $\mathrm{Ad}(a)L_1\cap L_2=W(\tilde{\Sigma})x_0$ holds if and only if the condition 
$L_1'\cap L_2'=$\linebreak$\mathrm{Ad}(k_2)(W(\tilde{\Sigma})x_0)$ holds. 
Moreover, ${\mathfrak a}':=\mathrm{Ad}(k_2)\mathfrak{a}$ is a maximal abelian subspace of 
$\mathfrak{p}_1'\cap \mathfrak{p}_2'$. 
\end{enumerate}
\end{lemma}

We have $G=K_2AK_1$ by a theorem of Hermann (\cite[p.~293, Theorem~2.2]{Hermann}) . 
By Lemma~\ref{lem:equiv}, in order to prove Lemma~\ref{lem:intersect}, it is sufficient to show that 
there is a representative $(G,\theta_1',\theta_2')$ in each equivalent class $[(G,\theta_1,\theta_2)]$ 
such that 
\begin{enumerate}
\item[(i)] When $a$ runs through $A$, if $\mathrm{Ad}(a){\mathfrak p}_1\cap {\mathfrak p}_2$ is not abelian, then
$\mathrm{Ad}(a)L_1\cap L_2$ is not discrete. 
\item[(ii)] 
$W(\mathfrak{t})x_0\cap\mathfrak{a}=W(\tilde{\Sigma})x_0$ holds for any $x_0\in\mathfrak{a}$.
\end{enumerate}

In Subsection~\ref{subs:cst}, we will show that there exists a representative $(G,\theta_1,\theta_2)$ which has 
good properties in each equivalent class $[(G,\theta_1,\theta_2)]$ (Theorem~\ref{thm:can}). 
The assertions (i) and (ii) above will be shown for such a $(G,\theta_1,\theta_2)$ 
in Subsection 3.3 (see Lemmas~\ref{lem:finiteorder} and \ref{lem:B}).

\subsection{Compact symmetric triads}
\label{subs:cst}

In this subsection we prepare the foundations of compact symmetric triads, which will be used later. 

We retain the notation in the previous section. 
Let $(G,\theta_1,\theta_2)$ be a compact symmetric triad, where $G$ is a compact connected simple Lie group. 
Take a maximal abelian subspace $\mathfrak{a}$ of ${\mathfrak p}_1\cap {\mathfrak p}_2$. 
We call $H\in\mathfrak{a}$ a {\it regular point} of $(G,\theta_1,\theta_2)$ with respect to $\mathfrak{a}$, if 
$(\mathrm{Ad}(\exp H){\mathfrak p}_1)\cap {\mathfrak p}_2=\mathfrak{a}$ holds. 
Let ${\mathfrak a}_r$ denote the set of regular points. 
A connected component of $\mathfrak{a}_r$ is called a {\it cell}. 
See Theorem~\ref{thm:discrete} below for the motivation of the definitions of a regular point and $\mathfrak{a}_r$.

\begin{lemma}
Let $\mathfrak{a}$ and $\mathfrak{a}'$ be two maximal abelian subspaces of ${\mathfrak p}_1\cap {\mathfrak p}_2$. 
Then there exists $k\in F(\theta_1,G)\cap F(\theta_2,G)$ such that
$\mathfrak{a}'=\mathrm{Ad}(k)\mathfrak{a}$. 
Furthermore, $\mathfrak{a}'_r=\mathrm{Ad}(k)\mathfrak{a}_r$ holds. 
\end{lemma}
\begin{proof}
First, we show the first part. 
Let $G_{12}$ denote the identity component of \linebreak$F(\theta_1\theta_2,G)$. 
Then $(G_{12},G_{12}\cap F(\theta_1,G)\cap F(\theta_2,G))$ is a compact symmetric pair. 
The canonical decomposition of the Lie algebra $\mathfrak{g}_{12}$ of 
$G_{12}$ is given by 
$\mathfrak{g}_{12}=(\mathfrak{k}_1\cap {\mathfrak k}_2)\oplus (\mathfrak{p}_1\cap {\mathfrak p}_2)$. 
Applying a general theory of compact symmetric pair (\cite[p.~247, Lemma~6.3, (iii)]{Helgason}) to 
 $(G_{12},G_{12}\cap F(\theta_1,G)\cap F(\theta_2,G))$, we get the assertion. 

Next, we show the latter part. 
Since ${\mathfrak p}_i$ is $\mathrm{Ad}(k)$-invariant, we have 
\begin{align*}
{\mathfrak a}'_r
&=\{X\in {\mathfrak a}'\mid (\mathrm{Ad}(\exp X){\mathfrak p}_1)\cap {\mathfrak p}_2
={\mathfrak a}'\}\\
&=\{\mathrm{Ad}(k)H\mid 
H\in {\mathfrak a}, 
(\mathrm{Ad}(\exp (\mathrm{Ad}(k)H)){\mathfrak p}_1)\cap {\mathfrak p}_2
=\mathrm{Ad}(k){\mathfrak a}\}\\
&=\mathrm{Ad}(k)\mathfrak{a}_r.
\end{align*}
\end{proof}

The following lemma is obvious.

\begin{lemma} \label{lem:regular}
Retain the notation as in Lemma~\ref{lem:equiv}. 
The set ${\mathfrak a}_r'$ of regular points with respect to the maximal abelian subspace 
${\mathfrak a}'$ of ${\mathfrak p}_1'\cap {\mathfrak p}_2'$ is given as follows in cases 
(1) and (2) of Lemma~\ref{lem:equiv}, where we set $a=\exp Y\; (Y\in\mathfrak{a})$ 
in case (2) of Lemma~\ref{lem:equiv}. 
\[
{\mathfrak a}_r'=\left\{
\begin{array}{ll}
\varphi({\mathfrak a}_r) & \mbox{(Lemma~\ref{lem:equiv}, (1))},\\
\mathrm{Ad}(k_2)\{H\in\mathfrak a\mid H+Y\in \mathfrak{a}_r\} 
& \mbox{(Lemma~\ref{lem:equiv}, (2)).}
\end{array}
\right.
\]
\end{lemma}

Here, we summarize the well-known results about compact symmetric triads 
in order to use the proof of Theorem~\ref{thm:int}.
For a while, we assume that the order $\mathrm{ord}(\theta_1\theta_2)$ of $\theta_1\theta_2$ is finite, and 
set $n:=\mathrm{ord}(\theta_1\theta_2)\in \mathbb{N}$. 
In this case, we can use the third author's results in \cite{Ohno}.
The complex linear extension of $\theta_1\theta_2$ to ${\mathfrak g}^{\mathbb C}$ is 
diagonalizable, and its eigenvalues are $n$-roots of $1$. 
For $\epsilon \in U(1)$, we define the complex subspace $\mathfrak{g}(\mathfrak{a},\alpha,\epsilon)$ of 
$\mathfrak{g}(\mathfrak{a},\alpha)$ by 
\[
\mathfrak{g}(\mathfrak{a},\alpha,\epsilon)=\{X\in \mathfrak{g}(\mathfrak{a},\alpha)\mid \theta_1\theta_2 X=\epsilon X\}. 
\]
Since $\mathfrak{g}(\mathfrak{a},\alpha)$ is $\theta_1\theta_2$-invariant, we have 
\begin{equation}
\label{eqn:decomp0}
\mathfrak{g}(\mathfrak{a},\alpha)=\sum_{\epsilon\in U(1)}\mathfrak{g}(\mathfrak{a},\alpha,\epsilon). 
\end{equation}
Moreover, we get 
\begin{equation}
\label{eqn:alphep}
\overline{\mathfrak{g}(\mathfrak{a},\alpha,\epsilon)}
=\theta_i(\mathfrak{g}(\mathfrak{a},\alpha,\epsilon))
=\mathfrak{g}(\mathfrak{a},-\alpha,\epsilon^{-1}),\quad 
[\mathfrak{a},\mathfrak{g}(\mathfrak{a},\alpha,\epsilon)]\subset 
\mathfrak{g}(\mathfrak{a},\alpha,\epsilon).
\end{equation}

For $\epsilon\in U(1)$, we put
\[
\Sigma_\epsilon=\{\alpha\in\tilde{\Sigma}\mid \mathfrak{g}(\mathfrak{a},\alpha,\epsilon)\not=\{0\}\}.
\]
Then we have 
\[
\tilde{\Sigma}=\bigcup_{\epsilon\in U(1)}\Sigma_\epsilon.
\]

Define the complex subspace $U(\alpha,\epsilon)$ of ${\mathfrak{g}}^{\mathbb C}$ by
\[
U(\alpha,\epsilon)=\mathfrak{g}(\mathfrak{a},\alpha,\epsilon)\oplus \mathfrak{g}(\mathfrak{a},-\alpha,\epsilon^{-1}).
\]
Then we have
$\overline{U(\alpha,\epsilon)}=\theta_i(U(\alpha,\epsilon))=U(\alpha,\epsilon)$
and $[\mathfrak{a},U(\alpha,\epsilon)]\subset U(\alpha,\epsilon)$ by (\ref{eqn:alphep}). 
Define the subspace $V(\alpha,\epsilon)$ of $\mathfrak{g}$ by 
\[
V(\alpha,\epsilon)=U(\alpha,\epsilon)\cap \mathfrak{g}.
\]
Then $\theta_i(V(\alpha,\epsilon))=V(\alpha,\epsilon)$ and 
$[\mathfrak{a},V(\alpha,\epsilon)]\subset V(\alpha,\epsilon)$ hold. 
Set 
\[
\mathfrak{g}(0)=\{X\in\mathfrak{g}^{\mathbb C}\mid [X,\mathfrak{a}]=\{0\}\}. 
\]
Then by \cite[p.~66, the third line from the bottom]{Ohno}, we have 
\begin{equation}
\mathfrak g
=(\mathfrak{g}(0)\cap \mathfrak g)
\oplus\sum_{\alpha\in\tilde{\Sigma}^+}\sum_{\epsilon\in U(1)}V(\alpha,\epsilon),
\label{eqn:frakg}
\end{equation}
where $\tilde{\Sigma}^+$ denote the set of positive roots in  $\tilde{\Sigma}$. 
We set $\Sigma^+_\epsilon=\tilde{\Sigma}^+\cap \Sigma_\epsilon$ for $\epsilon\in U(1)$. 

Since $\mathfrak g(0)\cap\mathfrak g$ and $V(\alpha,\epsilon)$ are 
$\theta_i$-invariant for $i=1,2$, by using (\ref{eqn:frakg}), we have 
\begin{align}
& {\mathfrak p}_i
=(\mathfrak g(0)\cap {\mathfrak p}_i)\oplus
\sum_{\epsilon\in U(1)}\sum_{\alpha\in \Sigma^+_\epsilon}
(V(\alpha,\epsilon)\cap {\mathfrak p}_i),\label{eqn:m1}\\
& {\mathfrak k}_i
=(\mathfrak g(0)\cap {\mathfrak k}_i)\oplus
\sum_{\epsilon\in U(1)}\sum_{\alpha\in \Sigma^+_\epsilon}
(V(\alpha,\epsilon)\cap {\mathfrak k}_i)\label{eqn:k1}.
\end{align}

Note that 
\begin{align*}
V(\alpha,\epsilon)&=(V(\alpha,\epsilon)\cap\mathfrak{k}_1)\oplus (V(\alpha,\epsilon)\cap\mathfrak{p}_1)\\
&=(V(\alpha,\epsilon)\cap\mathfrak{k}_2)\oplus (V(\alpha,\epsilon)\cap\mathfrak{p}_2),
\end{align*}
since $V(\alpha,\epsilon)$ is $\theta_i$-invariant for $i=1,2$.
We know by \cite{Ohno} that 
\begin{align*}
&\dim (V(\alpha,\epsilon)\cap {\mathfrak p}_1)
=\dim (V(\alpha,\epsilon)\cap {\mathfrak k}_1)\\
&=\dim (V(\alpha,\epsilon)\cap {\mathfrak p}_2)
=\dim (V(\alpha,\epsilon)\cap {\mathfrak k}_2).
\end{align*}
Set $m(\alpha,\epsilon)=\dim (V(\alpha,\epsilon)\cap {\mathfrak p}_1)$. 
For an eigenvalue $\epsilon\in U(1)$ of $\theta_1\theta_2$, we define $\varphi_\epsilon\in (-\pi/2,\pi/2]$ 
by $\epsilon =e^{2\sqrt{-1}\varphi_\epsilon}$.

\begin{lemma} \cite[Lemma~3]{Ohno} \label{lem:3}
For each $\epsilon\in U(1)$ and $\alpha\in \Sigma^+_\epsilon$, 
there exist orthonormal bases $\{X^\epsilon_{\alpha, i}\}_{i=1}^{m(\alpha,\epsilon)}$ and 
$\{Y^\epsilon_{\alpha,i}\}_{i=1}^{m(\alpha,\epsilon)}$ of $V(\alpha,\epsilon)\cap {\mathfrak k}_1$ and 
$V(\alpha,\epsilon)\cap {\mathfrak p}_1$ respectively which satisfy the following conditions (1) and (2). 
\begin{enumerate}
\item[(1)] For any $H\in\mathfrak a$, 
$$
[H,X^\epsilon_{\alpha,i}]=\langle\alpha,H\rangle Y^\epsilon_{\alpha,i},\;
[H,Y^\epsilon_{\alpha,i}]=-\langle\alpha,H\rangle X^\epsilon_{\alpha,i},\; 
[X^\epsilon_{\alpha,i},Y^\epsilon_{\alpha,i}]=\alpha
$$
hold. 
Moreover, for $H\in\mathfrak a$, 
\begin{align*}
& \mathrm{Ad}(\exp H)X^\epsilon_{\alpha,i}
=\cos\langle\alpha,H\rangle X^\epsilon_{\alpha,i}
+\sin \langle\alpha,H\rangle Y^\epsilon_{\alpha,i},\\
& \mathrm{Ad}(\exp H)Y^\epsilon_{\alpha,i}
=-\sin\langle\alpha,H\rangle X^\epsilon_{\alpha,i}
+\cos\langle\alpha,H\rangle Y^\epsilon_{\alpha,i}
\end{align*}
hold. 
\item[(2)] For $H_\epsilon=H_{\epsilon,\alpha}\in\mathfrak a$ which satisfies 
$\langle\alpha,H_\epsilon\rangle=\varphi_\epsilon$, we set
\[
\tilde{X}_{\alpha,i}^\epsilon=\mathrm{Ad}(\exp H_\epsilon)^{-1}X_{\alpha,i}^\epsilon, 
\quad 
\tilde{Y}_{\alpha,i}^\epsilon=\mathrm{Ad}(\exp H_\epsilon)^{-1}Y_{\alpha,i}^\epsilon.
\]
\end{enumerate}
Then $\{\tilde{X}^\epsilon_{\alpha,i}\}_{i=1}^{m(\alpha,\epsilon)}$
and $\{\tilde{Y}^\epsilon_{\alpha,i}\}_{i=1}^{m(\alpha,\epsilon)}$ are orthonormal bases of 
$V(\alpha,\epsilon)\cap {\mathfrak k}_2$ and $V(\alpha,\epsilon)\cap {\mathfrak p}_2$ respectively, 
and for any $H\in\mathfrak a$, 
\[
[H,\tilde{X}^\epsilon_{\alpha,i}]=\langle\alpha,H\rangle \tilde{Y}^\epsilon_{\alpha,i},\quad 
[H,\tilde{Y}^\epsilon_{\alpha,i}]=-\langle\alpha,H\rangle \tilde{X}^\epsilon_{\alpha,i},\quad 
[\tilde{X}^\epsilon_{\alpha,i},\tilde{Y}^\epsilon_{\alpha,i}]=\alpha
\]
hold. 
Moreover, for $H\in\mathfrak a$, 
\begin{align*}
& \mathrm{Ad}(\exp H)\tilde{X}^\epsilon_{\alpha,i}
=\cos\langle\alpha,H\rangle \tilde{X}^\epsilon_{\alpha,i}
+\sin \langle\alpha,H\rangle \tilde{Y}^\epsilon_{\alpha,i},\\
& \mathrm{Ad}(\exp H)\tilde{Y}^\epsilon_{\alpha,i}
=-\sin\langle\alpha,H\rangle \tilde{X}^\epsilon_{\alpha,i}
+\cos\langle\alpha,H\rangle \tilde{Y}^\epsilon_{\alpha,i}
\end{align*}
hold. 
\end{lemma}

\begin{lemma} \label{fact:1}
For $H\in\mathfrak a$, put $a=\exp H$. 
For $\epsilon\in U(1)$ and $\alpha\in \Sigma^+_\epsilon$, 
the following equation holds: 
\[
V(\alpha,\epsilon)\cap \mathrm{Ad}(a){\mathfrak p}_1\cap {\mathfrak p}_2
=\left\{
\begin{array}{ll}
V(\alpha,\epsilon)\cap {\mathfrak p}_2 & 
(\langle\alpha,H\rangle+\varphi_\epsilon \in\pi\mathbb{Z}),\\
\{0\} & (\langle\alpha,H\rangle+\varphi_\epsilon \not\in\pi\mathbb{Z}).
\end{array}
\right.
\]
\end{lemma}
\begin{proof}
Since $[\mathfrak a,V(\alpha,\epsilon)]$ is a subset of $V(\alpha,\epsilon)$, we have
\[
V(\alpha,\epsilon)\cap \mathrm{Ad}(a){\mathfrak p}_1
=\mathrm{Ad}(a)(V(\alpha,\epsilon)\cap {\mathfrak p}_1).
\]
By Lemma~\ref{lem:3}, we have
\begin{align*}
\mathrm{Ad}(a)Y_{\alpha,i}^\epsilon
&=\mathrm{Ad}(\exp(H+H_\epsilon))\tilde{Y}_{\alpha,i}^\epsilon\\
&=-\sin \langle\alpha,H+H_\epsilon\rangle \tilde{X}_{\alpha,i}^\epsilon
+\cos \langle\alpha,H+H_\epsilon\rangle \tilde{Y}_{\alpha,i}^\epsilon,
\end{align*}
which implies that
\[
\mathrm{Ad}(a)Y_{\alpha,i}^\epsilon\in {\mathfrak p}_2
\Leftrightarrow \sin (\langle\alpha,H\rangle +\varphi_\epsilon)=0
\Leftrightarrow \langle\alpha,H\rangle +\varphi_\epsilon\in\pi\mathbb{Z}.
\]
In this case, we have $\mathrm{Ad}(a)Y_{\alpha,i}^\epsilon=\pm \tilde{Y}_{\alpha,i}^\epsilon$. 
Thus we get the assertion. 
\end{proof}

\begin{lemma}\label{fact:2}
For $H\in\mathfrak a$, set $a=\exp H$. 
Then the following holds:
\begin{equation}
\label{eqn:decomp}
\mathrm{Ad}(a){\mathfrak p}_1\cap {\mathfrak p}_2
=\mathfrak a\oplus \sum_{\epsilon\in U(1)}
\sum_{\substack{\alpha\in \Sigma_\epsilon^+\\ \langle\alpha,H\rangle +\varphi_\epsilon\in\pi
\mathbb{Z}}}(V(\alpha,\epsilon)\cap {\mathfrak p}_2).
\end{equation}
Furthermore, the following three conditions are equivalent.
\begin{enumerate}
\item[(1)] $H$ is a regular point of $(G,\theta_1,\theta_2)$ with respect to $\mathfrak{a}$. 
\item[(2)] $\langle\alpha,H\rangle+\varphi_\epsilon\not\in\pi\mathbb{Z}$ for any $\epsilon\in U(1)$ and $\alpha\in \Sigma_\epsilon^+$. 
\item[(3)] $\langle\alpha,H\rangle+\varphi_\epsilon\not\in\pi\mathbb{Z}$ for any $\epsilon\in U(1)$ and $\alpha\in \Sigma_\epsilon$. 
\end{enumerate}
\end{lemma}
\begin{proof}
By using (\ref{eqn:m1}), Lemma~\ref{fact:1} and the fact that $\mathrm{Ad}(a)$ acts on 
$\mathfrak{g}(0)\cap {\mathfrak p}_2$ trivially, we have 
\[
\mathrm{Ad}(a){\mathfrak p}_1\cap {\mathfrak p}_2
=(\mathfrak{g}(0)\cap {\mathfrak p}_1\cap {\mathfrak p}_2)
\oplus \sum_{\epsilon\in U(1)}
\sum_{\substack{\alpha\in \Sigma_\epsilon^+\\ \langle\alpha,H\rangle +\varphi_\epsilon\in\pi
\mathbb{Z}}}(V(\alpha,\epsilon)\cap {\mathfrak p}_2).
\]
Since $\mathfrak{g}(0)\cap {\mathfrak p}_1\cap {\mathfrak p}_2$ equals $\mathfrak{a}$, we get (\ref{eqn:decomp}), which yields that 
(1) and (2) are equivalent. 
It follows from $\Sigma_\epsilon^+\subset\Sigma_\epsilon$ that $(3)\Rightarrow (2)$. 
In order to prove $(2)\Rightarrow(3)$, take $\alpha$ in $\Sigma_\epsilon-\Sigma_\epsilon^+$. 
Since $-\alpha$ is in $\Sigma^+_{\bar{\epsilon}}$, we have 
$-\langle\alpha,H\rangle+\varphi_{\bar{\epsilon}}\not\in\pi\mathbb{Z}$. 
Since $\varphi_\epsilon+\varphi_{\bar{\epsilon}}\in\pi\mathbb{Z}$, we have
$\langle\alpha,H\rangle+\varphi_\epsilon\not\in\pi\mathbb{Z}$, 
which implies the assertion. 
\end{proof}

By Lemma~\ref{fact:2}, ${\mathfrak a}_r$ is an open dense subset of $\mathfrak{a}$.

\begin{lemma}\label{lem:Weyl}
The following relation holds:
\begin{equation}
W(\tilde{\Sigma})\subset 
\{\mathrm{Ad}(g)|_{\mathfrak a}\mid \mathrm{Ad}(g)(\mathfrak{a})=\mathfrak{a},\; g\in G\}.
\label{eqn:Weyl}
\end{equation}
\end{lemma}
\begin{proof}
By Lemma~\ref{lem:3}, for any $X\in\mathfrak{a}$, we have 
\begin{equation}
\label{eqn:weyl}
\mathrm{Ad}(\exp t\tilde{X}_{\alpha,i}^\epsilon)X
=X+(\cos (t\|\alpha\|)-1)\frac{\langle\alpha,X\rangle}{\|\alpha\|^2}\alpha
-\sin (t\|\alpha\|)\frac{\langle\alpha,X\rangle}{\|\alpha\|^2}\tilde{Y}_{\alpha,i}^\epsilon.
\end{equation}
Take $t\in\mathbb{R}$ with $\cos (t\|\alpha\|)=-1$. 
Then $\mathrm{Ad}(\exp t\tilde{X}_{\alpha,i}^\epsilon)X=X-2\frac{\langle\alpha,X\rangle}{\|\alpha\|^2}\alpha$. 
Thus we get the assertion.
\end{proof}

\begin{definition}\rm \label{dfn:canonical}\cite[Definition~6.1]{BI}
A compact symmetric triad $(G,\theta_1,\theta_2)$ is said to be {\it canonical}, if there exists a 
maximal abelian subalgebra $\mathfrak{t}$ of $\mathfrak{g}$ which satisfies the following conditions (C1) and (C2): 
\begin{enumerate}
\item[(C1)]  $\mathfrak{t}$ is {\it quasi-canonical} with respect to $(G,\theta_1,\theta_2)$, that is, 
\begin{enumerate}
\item[(i)] $\mathfrak{a}_i:=\mathfrak{t}\cap {\mathfrak p}_i$ is a maximal abelian subspace of ${\mathfrak p}_i$ for $i=1,2$.
In particular, $[{\mathfrak a}_1,{\mathfrak a}_2]=\{0\}$ holds, and $\mathfrak{t}$ is $\theta_i$-invariant.
\item[(ii)] Let $\Delta$ denote the root system of $\mathfrak{g}$ with respect to $\mathfrak{t}$. 
Then there exists a linear ordering $<$ on $\mathfrak{t}$ such that $-\theta_i(\alpha)>0$ holds for 
$\alpha\in \Delta^+-\Delta^i_0$ ($i=1,2$), where $\Delta^+$ denotes the set of positive roots in $\Delta$ 
with respect to $<$, and $\Delta_0^i=\Delta \cap {\mathfrak k}_i$. 
\end{enumerate}
\item[(C2)] $\mathrm{ord}(\theta_1\theta_2)
=\mathrm{ord}(\theta_1\theta_2|\mathfrak{t})<\infty$. 
\end{enumerate}
Then, $\mathfrak{t}$ is said to be {\it canonical} with respect to $(G,\theta_1,\theta_2)$. 
\end{definition}

\begin{theorem} \cite[Theorem~6.12, (1)]{BI}\label{thm:BI}
Let $\mathfrak{t}$ be a canonical maximal abelian subalgebra of $\mathfrak{g}$ with respect to a 
canonical compact symmetric triad $(G,\theta_1,\theta_2)$. 
Then, $\mathfrak{a}:=\mathfrak{t}\cap (\mathfrak{p}_1\cap \mathfrak{p}_2)$ is a maximal abelian subspace 
of $\mathfrak{p}_1\cap \mathfrak{p}_2$. 
\end{theorem}

A liner ordering $<$ on $\mathfrak{t}$ induces a linear ordering $<$ on $\mathfrak{a}$. 
Let $\tilde{\Sigma}^+$ denote the set of positive roots in  $\tilde{\Sigma}$. 
We put 
\begin{align}
& \mathfrak{a}_+=\{X\in \mathfrak{a}\mid \langle \lambda,X\rangle \geq 0\quad (\lambda\in\tilde{\Sigma}^+)\},\label{eqn:a+}\\
& \mathfrak{t}_+=\{X\in \mathfrak{t}\mid \langle \alpha,X\rangle \geq 0\quad (\alpha\in\Delta^+)\}\label{eqn:t+}.
\end{align}

In the case when $(G,\theta_1,\theta_2)$ is canonical, we take $\mathfrak{t}$ and a linear ordering $<$ on $\mathfrak{t}$ 
as in Definition~\ref{dfn:canonical}.

\begin{lemma}\label{lem:okuda-baba}
Assume that a compact symmetric triad $(G,\theta_1,\theta_2)$ is canonical, and 
a maximal abelian subalgebra $\mathfrak{t}$ of $\mathfrak{g}$ is canonical with respect to $(G,\theta_1,\theta_2)$. 
Let $\;\bar{ }: \mathfrak{t}\rightarrow \mathfrak{a}$ denote the orthogonal projection. 
Then 
$\tilde{\Sigma}^+=\{\bar{\alpha}\mid \alpha\in\Delta^+\}-\{0\}$ holds. 
In particular, ${\mathfrak a}_+=\mathfrak{a}\cap {\mathfrak t}_+$ holds. 
\end{lemma}
\begin{proof}
It is clear that $\tilde{\Sigma}=\{\bar{\alpha}\mid \alpha\in\Delta\}-\{0\}$. 
Since $\mathfrak{t}$ is canonical, $\mathfrak{t}$ is $\theta_i$-invariant ($i=1,2$). 
Define the subspace ${\mathfrak t}^{\theta_1\theta_2}$ of $\mathfrak{t}$ by
\[
{\mathfrak t}^{\theta_1\theta_2}
=\{H\in\mathfrak{t}\mid \theta_1\theta_2(H)=H\}
=\{H\in\mathfrak{t}\mid \theta_1(H)=\theta_2(H)\}.
\]
Then $\mathfrak{a}$ is a subspace of ${\mathfrak t}^{\theta_1\theta_2}$. 
Let $\pi_1:\mathfrak{t}\rightarrow {\mathfrak t}^{\theta_1\theta_2}$ denote the orthogonal projection. 
Set $n:=\mathrm{ord}(\theta_1\theta_2)(<\infty)$. 
We show that 
\begin{equation}
\label{eqn:proj}
\pi_1(H)=\frac{1}{n}\sum_{k=0}^{n-1} (\theta_1\theta_2)^k (H)
\in {\mathfrak t}^{\theta_1\theta_2}.
\end{equation}
In fact, we have
\begin{align*}
&\langle H-\frac{1}{n}\sum_{k=0}^{n-1} (\theta_1\theta_2)^k (H), 
\frac{1}{n}\sum_{l=0}^{n-1} (\theta_1\theta_2)^l (H)\rangle \\
&=\frac{1}{n}\langle H,\sum_{l=0}^{n-1} (\theta_1\theta_2)^l (H)\rangle 
-\frac{1}{n^2}\sum_{k, l=0}^{n-1}\langle (\theta_1\theta_2)^k(H),(\theta_1\theta_2)^l(H)\rangle .
\end{align*}
Here, since
\begin{align*}
\sum_{k,l=0}^{n-1}\langle (\theta_1\theta_2)^k(H),(\theta_1\theta_2)^l(H)\rangle 
&=\sum_{k,l=0}^{n-1}\langle (\theta_1\theta_2)^k(H),(\theta_1\theta_2)^{n-l}(H)\rangle \\
&=\sum_{k,l=0}^{n-1}\langle (\theta_1\theta_2)^{k+l}(H),H\rangle \\
&=n \sum_{p=0}^{n-1}\langle (\theta_1\theta_2)^p (H),H\rangle ,
\end{align*}
we get 
\[
\langle H-\frac{1}{n}\sum_{k=0}^{n-1} (\theta_1\theta_2)^k (H), 
\frac{1}{n}\sum_{l=0}^{n-1} (\theta_1\theta_2)^l (H)\rangle =0.
\]
Thus we get (\ref{eqn:proj}). 
Similarly, we have 
\begin{equation}
\label{eqn:proj2}
\pi_1(H)=\frac{1}{n}\sum_{k=0}^{n-1} (\theta_2\theta_1)^k (H).
\end{equation}
Let $\pi_2:{\mathfrak t}^{\theta_1\theta_2}\rightarrow\mathfrak{a}$ denote the orthogonal projection. 
Then, for $H\in {\mathfrak t}^{\theta_1\theta_2}$, we have 
\begin{equation}
\label{eqn:p2}
\pi_2(H)=\frac{1}{2}(H-\theta_1H)=\frac{1}{2}(H-\theta_2 H). 
\end{equation}
By using (\ref{eqn:proj}), (\ref{eqn:proj2}) and (\ref{eqn:p2}), for $H\in \mathfrak{t}$ we have
\begin{align}
\bar{H}&=\pi_2\pi_1(H)
=\frac{1}{2n}\sum_{k=0}^{n-1}\{(\theta_1\theta_2)^k(H) 
-\theta_1(\theta_1\theta_2)^k(H)\}\nonumber\\
&=\frac{1}{2n}\sum_{k=0}^{n-1}\{(\theta_1\theta_2)^k(H) 
-(\theta_1\theta_2)^k\theta_1(H)\}.\label{eqn:bar}
\end{align}
Take $\alpha\in\Delta^+$ with $\bar{\alpha}\not=0$. 
Since $\alpha\not\in \Delta_0^1$ and $\alpha\not\in\Delta_0^2$, 
we have $-\theta_i(\alpha)\in \Delta^+$ by Definition~\ref{dfn:canonical}, (ii). 
Since $\overline{-\theta_i(\alpha)}=\bar{\alpha}\not=0$, we have 
$\theta_j\theta_i(\alpha)=(-\theta_j)(-\theta_i)(\alpha)\in\Delta^+$. 
Thus, we have $(\theta_1\theta_2)^k(\alpha)>0$ and $-(\theta_1\theta_2)^k\theta_1(\alpha)>0$ by induction, which implies that 
$\bar{\alpha}>0$ by (\ref{eqn:bar}). 
Thus we get $\tilde{\Sigma}^+\supset \{\bar{\alpha}\mid \alpha\in\Delta^+\}-\{0\}$. 
Conversely, we take $\lambda\in\tilde{\Sigma}^{+}$.
There exists $\alpha\in\Delta$ such that $\lambda=\bar{\alpha}$ holds.
It is sufficient to show that $\alpha$ is in $\Delta^{+}$.
For contradiction, suppose that $\alpha$ is in $\Delta-\Delta^{+}$.
Then $-\alpha$ is in $\Delta^{+}$, and $\overline{-\alpha}\neq 0$ holds.
By the argument above, we obtain $\overline{-\alpha}>0$.
But this contradicts that $\lambda$ is positive.
Hence $\lambda$ is in $\{\bar{\alpha}\mid \alpha\in\Delta^+\}-\{0\}$.
From the above we get $\tilde{\Sigma}^+=\{\bar{\alpha}\mid \alpha\in\Delta^+\}-\{0\}$. 

By using this, we have 
\begin{align*}
\mathfrak{a}\cap {\mathfrak t}_+
&=\{X\in\mathfrak{a}\mid \langle\alpha,X\rangle \geq 0\quad (\alpha\in\Delta^+)\}\\
&=\{X\in\mathfrak{a}\mid 
\langle\bar{\alpha},X\rangle \geq 0\quad (\alpha\in\Delta^+)\}\\
&=\{X\in\mathfrak{a}\mid 
\langle\lambda,X\rangle \geq 0\quad (\lambda\in\tilde{\Sigma}^+)\}\\
&={\mathfrak a}_+.
\end{align*}
\end{proof}

The following theorem demonstrates the importance of canonical forms of compact symmetric triads. 

\begin{theorem} \label{thm:can}\cite[Theorem~6.6]{BI}
For any compact symmetric triad $(G,\theta_1,\theta_2)$, there exists a 
canonical compact symmetric triad $(G,\theta_1,\theta_2')$ which satisfies 
$(G,\theta_1,\theta_2')\sim $\linebreak$(G,\theta_1,\theta_2)$. 
\end{theorem}

\subsection{Proof}
\label{subsec:proof}

After some preparations, we give the proof of Theorem~\ref{thm:int}. 

\begin{lemma} \label{lem:finiteorder}
Assume that the order of $\theta_1\theta_2$ is finite. 
If $\mathrm{Ad}(a){\mathfrak p}_1\cap {\mathfrak p}_2$ is not abelian for some $a=\exp H$ with $H\in\mathfrak{a}$, 
then $\mathrm{Ad}(a)L_1\cap L_2$ is not discrete. 
\end{lemma}

\begin{proof}
Since $\mathrm{Ad}(a)\mathfrak{p}_1\cap \mathfrak{p}_2$ is not abelian, 
there exist $\epsilon \in U(1)$ and $\alpha\in \Sigma^+_\epsilon$ with 
$\langle\alpha,H\rangle +\varphi_\epsilon \in\pi\mathbb{Z}$ such that 
$V(\alpha,\epsilon)\cap {\mathfrak p}_2\not=\{0\}$ by (\ref{eqn:decomp}). 
By (2) of Lemma~\ref{lem:tildeSigma},
$\mathfrak{a}$ is equal to $\mathrm{span}\{W(\tilde{\Sigma})x_0\}$.
Thus, there exists $\sigma \in W(\tilde{\Sigma})$ such that
$\langle\alpha ,\sigma x_0\rangle\not=0$.
Set $X:=\sigma x_0$. 
Then $X$ is in $\mathfrak{a}\cap M$ by Lemma~\ref{lem:Weyl} and $\langle\alpha ,X\rangle\not=0$.
For the $X$ above, the right-hand side of (\ref{eqn:weyl}) is an element of $\mathrm{Ad}(a)\mathfrak{p}_1\cap \mathfrak{p}_2$ by (\ref{eqn:decomp}). 
On the other hand, the left-hand side of (\ref{eqn:weyl}) is an element of $M$, and 
the both sides of (\ref{eqn:weyl}) is a curve of $\mathrm{Ad}(a)L_1\cap L_2$ by (\ref{eqn:intersection}). 
Since $\langle \alpha, X \rangle$ is not equal to $0$,
(\ref{eqn:weyl}) represents a circle in $\mathrm{Ad}(a)L_1 \cap L_2$,
and the intersection is not discrete.
\end{proof}

\begin{lemma} \label{lem:B}
Assume that $n:=\mathrm{ord}(\theta_1\theta_2)<\infty$ and 
$\mathfrak{a}_+\subset \mathfrak{t}_+$. 
Then $W(\mathfrak{t})x_0\cap \mathfrak{a}=W(\tilde{\Sigma})x_0$ holds for any $x_0\in\mathfrak{a}$. 
\end{lemma}
\begin{proof}
We have
$
W(\tilde{\Sigma})x_0\subset M\cap\mathfrak{a}=W(\mathfrak{t})x_0\cap\mathfrak{a}
$ by using Lemmas~\ref{lem:Mcapa} and \ref{lem:Weyl}.
By Lemma~\ref{lem:Weyl}, $M\cap \mathfrak{a}$ is invariant under the action of $W(\tilde{\Sigma})$. 
It suffices to show that $M\cap \mathfrak{a}$ is a single $W(\tilde{\Sigma})$-orbit. 
Note that ${\mathfrak a}_+$ and ${\mathfrak t}_+$ defined by (\ref{eqn:a+}) and (\ref{eqn:t+}) 
are closed fundamental domains of the $W(\tilde{\Sigma})$-action on $\mathfrak{a}$ 
and $W(\mathfrak{t})$-action on $\mathfrak{t}$, respectively. 
Since ${\mathfrak a}_+$ is a subset of ${\mathfrak t}_+$, we get, by using \cite[p.~293, Theorem 2.2]{Helgason}, that 
\begin{align*}
1&=\#(M\cap {\mathfrak t}_+)\\
&\geq \#(M\cap {\mathfrak a}_+)=\#\{\mbox{the $W(\tilde{\Sigma})$-orbits $\subset M\cap\mathfrak{a}$}\}\\
&\geq \#\{W(\tilde{\Sigma})x_0\}=1.
\end{align*}
Thus, we get the assertion. 
\end{proof}

From Theorem~\ref{thm:can}, there exists a compact symmetric triad $(G,\theta_1',\theta_2')$ 
which satisfies (i) and (ii) in Subsection~3.1. 
Thus, we complete the proof of Theorem~\ref{thm:int}. 

\section{Formulation of settings and its application}
\label{sec:main2}

By using Lemma~\ref{lem:equiv}, Theorems~\ref{thm:BI} and \ref{thm:can}, we get the following in a similar argument 
in the previous section. 

\begin{theorem}\label{thm:can2}
Let $G$ be a compact connected simple Lie group, and $(G,\theta_1,\theta_2)$ be a compact symmetric triad. 
For $x_0\in {\mathfrak p}_1\cap {\mathfrak p}_2-\{0\}$, set 
$M=\mathrm{Ad}(G)x_0$ and $L_i=\mathrm{Ad}(K_i)x_0$. 
Then there exists a canonical compact symmetric triad $(G,\theta_1',\theta_2')$, 
a canonical maximal abelian subalgebra $\mathfrak{t}'$ of $\mathfrak{g}$ with respect to $(G,\theta_1',\theta_2')$, 
and $H\in \mathfrak{a}', x_0'\in \mathfrak{a}'$ such that 
\[
M=\mathrm{Ad}(G)x_0',\quad L_1=\mathrm{Ad}(\exp H)\mathrm{Ad}(K_1')x_0',\quad 
L_2=\mathrm{Ad}(K_2')x_0',
\]
where we put ${\mathfrak a}'={\mathfrak t}'\cap {\mathfrak p}_1'\cap {\mathfrak p}_2'$.
\end{theorem}

Based on the theorem above, we can set as follows.  
Let $(G,\theta_1,\theta_2)$ be a canonical compact symmetric triad. 
For $x_0\in\mathfrak{a}-\{0\}$, set 
\[
M=\mathrm{Ad}(G)x_0,\quad K_i=F(\theta_i,G)_0,\quad L_i=\mathrm{Ad}(K_i)x_0.
\]
In order to study the intersections of two real flag manifolds in a complex flag manifold, it suffices to consider the intersection 
\begin{equation}
\label{eqn:intersection2}
\mathrm{Ad}(\exp H)L_1\cap L_2\qquad (H\in\mathfrak{a}).
\end{equation}

Theorem~\ref{thm:int} and (\ref{eqn:decomp}) imply the following: 

\begin{theorem} \label{thm:discrete}
Under the setting above, $\mathrm{Ad}(\exp H)L_1\cap L_2$ is discrete if and only if 
$H$ is a regular point of $(G,\theta_1,\theta_2)$ with respect to $\mathfrak{a}$. 
\end{theorem}

We show that we can replace $\mathfrak{a}$ in (\ref{eqn:intersection2}) to the closure $\bar{P}$ 
of a cell $P$ (see Remark~\ref{rem:closure} below). 
In order to prove this, we prepare the following. 

The group $\{(k,0)\mid k\in Z_{K_1\cap K_2}(\mathfrak{a})\}\cong Z_{K_1\cap K_2}(\mathfrak{a})$ is a normal subgroup 
of the group $\{(s,Y)\in N_{K_1}(\mathfrak{a})\ltimes \mathfrak{a}\mid \exp(-Y)s\in K_2\}$, where we set 
$Z_{K_1\cap K_2}(\mathfrak{a})=\{k\in K_1\cap K_2\mid \mathrm{Ad}(k)=1\; \mbox{on}\; \mathfrak{a}\}$. 
Matsuki (\cite{Matsuki}, \cite{Matsuki2}) defined the group $\tilde{J}$ by 
\[
\tilde{J}=\{(s,Y)\in N_{K_2}(\mathfrak{a})\ltimes \mathfrak{a}\mid \exp(-Y)s\in K_1\}/
\{(k,0)\mid k\in Z_{K_1\cap K_2}(\mathfrak{a})\}.
\]
Then $\tilde{J}$ acts effectively on $\mathfrak{a}$. 
Let $W(\tilde{\Sigma},\{\Sigma_{\epsilon}\}_{\epsilon\in U(1)})$ denote the subgroup of $O(\mathfrak{a})\ltimes\mathfrak{a}$ 
generated by 
\[
\bigcup_{\epsilon\in U(1)}\left\{
\left(
s_\alpha,2\frac{n\pi-\varphi_\epsilon}{\|\alpha\|^2}\alpha
\right)
\mid 
\alpha\in\Sigma_\epsilon, n\in\mathbb{Z}
\right\}. 
\]
Then $W(\tilde{\Sigma},\{\Sigma_{\epsilon}\}_{\epsilon\in U(1)})$ is a subgroup of $\tilde{J}$ (\cite{Ohno}). 

For two subsets $A,B\subset M$ are called {\it congruent} with each other, if there exists $g\in G$ such that $B=\mathrm{Ad}(g)A$. 

\begin{proposition}\label{pro:affweylgr}
If  the two points $H_1\in\mathfrak{a}$ and $H_2\in\mathfrak{a}$ are transformed with each other 
by the action of $\tilde{J}$, then  
$(\mathrm{Ad}(\exp H_1)L_1)\cap L_2$ and $(\mathrm{Ad}(\exp H_2)L_1)\cap L_2$ 
are congruent with each other.

The group $\tilde{J}$ maps any cell to a cell. 
The group $W(\tilde{\Sigma},\{\Sigma_{\epsilon}\})$ acts transitively on the set of cells.
For any fixed cell $P$, the following holds: 
\[
\mathfrak{a}=\bigcup_{\sigma\in W(\tilde{\Sigma},\{\Sigma_{\epsilon}\})}\sigma\bar{P}.
\]
\end{proposition}
\begin{proof}
By the assumption, there exists $[(k,Y)]\in \tilde{J}$ such that $H_2=\mathrm{Ad}(k)H_1+Y$. 
Then 
\[
\mathrm{Ad}(\exp H_2)=\mathrm{Ad}(k)\mathrm{Ad}(\exp H_1)
\mathrm{Ad}(k^{-1}\exp Y). 
\]
Since $k$ is an element of $K_2$, and $k^{-1}\exp Y$ is an element of $K_1$, we have 
\begin{equation}
\label{eqn:L1L2}
(\mathrm{Ad}(\exp H_2)L_1)\cap L_2
=\mathrm{Ad}(k)((\mathrm{Ad}(\exp H_1)L_1)\cap L_2)
\end{equation}
and
\begin{equation}
\label{eqn:P1P2}
\mathrm{Ad}(\exp H_2)\mathfrak{p}_1\cap \mathfrak{p}_2
=\mathrm{Ad}(k)((\mathrm{Ad}(\exp H_1)\mathfrak{p}_1)\cap \mathfrak{p}_2).
\end{equation}
Hence, $(\mathrm{Ad}(\exp H_1)L_1)\cap L_2$ and $(\mathrm{Ad}(\exp H_2)L_1)\cap L_2$ 
are congruent with each other by (\ref{eqn:L1L2}). 

By (\ref{eqn:P1P2}), $\tilde{J}$ maps any regular point to a regular point. 
Since the action of $\tilde{J}$ on $\mathfrak{a}$ is continuous, $\tilde{J}$ maps any cell to a cell. 

Since $W(\tilde{\Sigma},\{\Sigma_{\epsilon}\})$ is a subgroup of $\tilde{J}$, 
$W(\tilde{\Sigma},\{\Sigma_{\epsilon}\})$ also maps any cell to a cell. 
We show that the action of $W(\tilde{\Sigma},\{\Sigma_{\epsilon}\})$ is transitive. 
Let $P_1$ and $P_2$ be two cells and select $H_i\in P_i\;(i=1,2)$. 
Take $s_0\in W(\tilde{\Sigma},\{\Sigma_{\epsilon}\})$ such that 
\[
|H_1-s_0 H_2|=\min\{|H_1-sH_2|\mid s\in W(\tilde{\Sigma},\{\Sigma_{\epsilon}\})\}.
\]
Since $s_0$ maps any cell to a cell, $s_0P_2=P_1$ or $s_0P_2\cap P_1=\emptyset$ holds. 
If it were to $s_0P_2\cap P_1=\emptyset$, there would exist an interior point $H$ of the line segment $\overline{H_1s_0H_2}$ 
and $\epsilon\in U(1),\alpha\in\Sigma_\epsilon, n\in\mathbb{Z}$ such that 
$\langle\alpha,H\rangle +\varphi_\epsilon=n\pi$ by Lemma~\ref{fact:2}. 
Set $s:=(s_\alpha,2\frac{n\pi-\varphi_\epsilon}{\|\alpha\|^2})\in W(\tilde{\Sigma},\{\Sigma_{\epsilon}\})$. 
Then 
\begin{align*}
&\|H_1-s_0H_2\|^2-\|H_1-ss_0H_2\|^2\\
&=-\frac{4}{\|\alpha\|^4}(n\pi-\varphi_\epsilon -\langle\alpha,H_1\rangle )
(n\pi-\varphi_\epsilon -\langle\alpha,s_0H_2\rangle )>0, 
\end{align*}
which would be a contradiction to the choice of $s_0$. 
\end{proof}

\begin{remark}\rm \label{rem:closure}
We can replace $H\in \mathfrak{a}$ in (\ref{eqn:intersection2}) to $H\in \bar{P}$ by 
Proposition~\ref{pro:affweylgr}. 
\end{remark}

\begin{theorem}\label{thm:cell}
Let $(G,\theta_1,\theta_2)$ be a compact symmetric triad, which is not necessarily canonical. 
Then the closure of a cell defined by $(G,\theta_1,\theta_2)$ is a simplex. 
\end{theorem}
\begin{proof}
By Lemma~\ref{lem:regular} and Theorem~\ref{thm:can}, we can and do assume that $(G,\theta_1,\theta_2)$ is canonical. 
By the classification of the equivalence classes of compact symmetric triads, we have $\mathrm{ord}(\theta_1\theta_2)=1,2,3,4,6$. 
Since $\dim\mathfrak{a}$ is greater than or equal to one, we have $\mathrm{ord}(\theta_1\theta_2)\not=6$. 
When $\mathrm{ord}(\theta_1\theta_2)=1$, the assertion is well-known by the theory of symmetric spaces. 
When $\mathrm{ord}(\theta_1\theta_2)=2$, the assertion is known by \cite[Cor.~2.21]{ikawa2011}. 
When $\mathrm{ord}(\theta_1\theta_2) \geq 3$, we can verify the assertion by realizing canonical compact symmetric triads $(G,\theta_1,\theta_2)$ explicitly, and by giving a concrete description of the closure of their cells.
When $\mathrm{ord}(\theta_1\theta_2)=3$, the assertion will be proved in Section~\ref{sec:order3}. 
When $\mathrm{ord}(\theta_1\theta_2)=4$, we can prove the assertion by a direct calculation. 
But we omit the proof because it is too long. 
Thus, the assertion holds.
\end{proof}

Note that even if $H\in\mathfrak{a}$ is a regular point with respect to $ (G,\theta_1,\theta_2)$, 
it does not necessarily a regular point with respect to $(G,\theta_2,\theta_1)$ 
interchanging the role of $\theta_1$ and $\theta_2$. 
The following two conditions (1) and (2) are equivalent.
\begin{enumerate}
\item[(1)] $\mathfrak{a}_r$ is invariant under multiplication by $-1$.
\item[(2)] For any $H \in \mathfrak{a}$, $H$ is a regular point with respect to $(G, \theta_1, \theta_2)$ if and only if it is a regular point of $(G, \theta_2, \theta_1)$.
\end{enumerate}

\begin{theorem}\label{thm:-1inv}
Let $(G,\theta_1,\theta_2)$ be a compact symmetric triad which is not neccessarily canonical. 
There exists a canonical compact symmetric triad $(G,\theta_1',\theta_2')$ which is  equivalent to  
$(G,\theta_1,\theta_2)$ whose ${\mathfrak a}_r'$ is invariant under the multiplication by $-1$. 
\end{theorem}
\begin{proof}
By Theorem~\ref{thm:can}, we can assume that 
$(G,\theta_1,\theta_2)$ is canonical. 
In this case, we have $\mathrm{ord}(\theta_1\theta_2)=1,2,3,4$ since 
$\dim\mathfrak{a}\geq 1$. 
When $\mathrm{ord}(\theta_1\theta_2)=1,2$, it is known that 
${\mathfrak a}_r$ is invariant under the multiplication by $-1$. 

When $\mathrm{ord}(\theta_1\theta_2)=3$, the assertion will be proved in Section~\ref{sec:order3} (see (\ref{eq:regularpointsord3})). 
When $\mathrm{ord}(\theta_1\theta_2)=4$, we can prove the assertion by a direct calculation. 
But we omit the proof because it is too long.

\end{proof}

Even if $(G,\theta_1,\theta_2)$ is canonical, 
${\mathfrak a}_r$ is not necessarily invariant under the multiplication by $-1$. 
We will give such an example in Section~\ref{sec:order3} (see Example~\ref{ex:multiplicationby-1}).

Let $\mathfrak{g}$ be a compact simple Lie algebra. 
Modeled after Definition~\ref{dfn:Matsuki}, Matsuki defined an equivalence relation on the set 
\[
\{(\mathfrak{g},\theta_1,\theta_2)\mid \mbox{$\theta_1,\theta_2$ is two involutions on 
$\mathfrak{g}$}\}\]
 as follows: 

\begin{definition}[\cite{Matsuki2}]\rm \label{dfn:Matsuk2i}
$(\mathfrak{g},\theta_1,\theta_2)$ and $(\mathfrak{g},\theta_1',\theta_2')$ are {\it equivalent}, if 
there exist an automorphism $\varphi\in\mathrm{Aut}(\mathfrak{g})$ on $\mathfrak{g}$ and 
an inner automorphism $\tau\in\mathrm{Int}(\mathfrak{g})$ on $\mathfrak{g}$ such that 
\begin{equation}
\theta_1'=\tau\varphi\theta_1\varphi^{-1}\tau^{-1},\quad 
\theta_2'=\varphi\theta_2 \varphi^{-1}.
\end{equation}
\end{definition}

As far as we know, the classification of the pair $(M,L)$ of a complex flag manifold $M$ and its real flag manifold $L$ is not complete yet. 
The pair $(M,L)$ is constructed from the pair $(\mathfrak{g},\theta)$ of a compact simple Lie algebra $\mathfrak{g}$ and 
its involution $\theta$. 
The pair $(\mathfrak{g},\theta)$ determines a Satake diagram and vice versa. 
Hence the classification of $(M,L)$ can be done by using Satake diagrams in principle. 
Generalizing the viewpoint, we discuss the classification of complex flag manifolds and 
two real flag manifolds in principle. 
In order to do this, let $(G,\theta_1,\theta_2)$ and $(G',\theta_1',\theta_2')$ be two compact symmetric triads which satisfy 
$\mathfrak{g}=\mathrm{Lie}(G)=\mathrm{Lie}(G')$ and $\theta_i=\theta_i'$ on $\mathfrak{g}$. 
Then $\mathfrak{p}_i=\mathfrak{p}_i'$. 
For $x_0\in \mathfrak{p}_1\cap \mathfrak{p}_2$, we have 
$\mathrm{Ad}(G)x_0=\mathrm{Ad}(G')x_0\supset \mathrm{Ad}(K_i)x_0=\mathrm{Ad}(K_i)x_0'$. 
That is, the triple of the complex flag manifold and its two real flag manifolds is completely determined by 
the triple $(\mathfrak{g}, \theta_1, \theta_2)$ and $x_0\in \mathfrak{p}_1\cap \mathfrak{p}_2$. 
The set of complex flag manifolds and their two real flag manifolds are obtained by the equivalence classes in the sense of Definition~\ref{dfn:Matsuk2i}(see Theorem~\ref{thm:can2}).
In order to state the classification in principle, we explain the notion of double Satake diagrams, which 
is an extension of that of Satake diagrams and was introduced by \cite{BI}.  
For a canonical compact symmetric triad $(G,\theta_1,\theta_2)$,  
take a canonical maximal abelian subalgebra $\mathfrak{t}$ of $\mathfrak{g}$ 
with respect to $(G,\theta_1,\theta_2)$. 
Take a linear ordering $<$ on $\mathfrak{t}$ which satisfies the condition (C1) (ii) of Definition~\ref{dfn:canonical}.
And $\Pi$ denotes the set of simple roots of $\Delta$.
Then $\Pi$ determines the Satake diagram of $(G,\theta_i)$ with respect to $\mathfrak{t}$ for $i=1,2$. 
The pair $(S_1,S_2)$ is called a {\it double Satake diagram} of $(G,\theta_1,\theta_2)$. 
Conversely, we can reconstruct the equivalent class of Lie algebra $\mathfrak{g}$ of $G$ and 
two involutions $\theta_1$ and $\theta_2$ on 
$\mathfrak{g}$ from the double Satake diagram. 
See \cite[Theorem~5.12]{BI}.
Since a complex flag manifold and its two real flag manifolds are determined by the equivalence class of 
$(\mathfrak{g},\theta_1,\theta_2)$, 
$x_0\in\mathfrak{a}-\{0\}$ and $H\in\bar{P}$ (see Theorem~\ref{thm:can2}), a double Satake diagram $(S_1,S_2)$ and its additional data 
$H\in\bar{P}$ and $x_0\in\mathfrak{a}-\{0\}$ completely determine the complex flag manifold $M=\mathrm{Ad}(G)x_0$ and 
its two real flag manifolds $L_1=\mathrm{Ad}(\exp H)\mathrm{Ad}(K_1)x_0$ and $L_2=\mathrm{Ad}(K_2)x_0$. 
Since we know the set of all double Satake diagrams by \cite{BI}, 
the set of complex flag manifolds and their two real flag manifolds can be classified in principle.

\section{The case of $\mathrm{ord}(\theta_1\theta_2)=3$
}
\label{sec:order3}

When $\mathrm{ord}(\theta_1\theta_2)=3$,
$(\mathfrak{g},\mathfrak{k}_1,{\mathfrak k}_2)=(\mathfrak{o}(8), \mathfrak{o}(3)\oplus \mathfrak{o}(5),\kappa(\mathfrak{o}(3)\oplus \mathfrak{o}(5)))$
holds where $\kappa$ is an outer automorphism of $\mathfrak{o}(8)$ with order three (see Proposition~\ref{defofkappa} below for the precise definition).
In this section, we realize the canonical form of this symmetric triad and certain basis of \(\mathfrak{g}\) concretely.
\subsubsection*{The involutions and the canonical decompositions}
We regard eight dimensional Euclidean space $\mathbb{R}^8$ as octonions $\mathbb{O} = \sum_{i=0}^7 \mathbb{R} e_i$ 
and use its structure to construct automorphisms of the Lie algebra $\mathfrak{o}(8)$.
The following automorphisms are essentially given by \cite{Loos2}, \cite{Yokota}.
The structure of the octonion multiplication in the following calculation is based on \cite{Yokota}.
We consider two conjugations  
\begin{align*}
\bar{x} &= x_0 e_0 - \sum_{i=1}^7 x_i e_i , \\
S(x) &= \sum_{i=0}^3 x_i e_i - \sum_{i=4}^7 x_i e_i
\end{align*}
for $x = \sum x_i e_i \in \mathbb{O}$. 
Then we define two involutions on $\mathfrak{o}(8)$ as follows: 
\begin{align*}
\pi^L (X)(x) &= \overline{X(\bar{x})},  \\
\sigma^{L}(X)(x) &= S(X(S(x))) \quad (X \in \mathfrak{o}(8), x \in \mathbb{O}).
\end{align*}
We can see that 
$\pi^L = \mathrm{Ad}(I_{1,7})$, 
$\sigma^{L}=\mathrm{Ad}(I_{4,4})$.
For $0\leq i < j\leq 7$, we define 
\begin{align*}
G_{ij}(x) &= (x, e_j)e_i - (x, e_i)e_j,  \\
F_{ij}(x) &= \frac{1}{2} (e_i(\bar{e_j}x))  \quad (x \in \mathbb{O}).
\end{align*}
Then, $\{G_{ij}\}_{0\leq i < j\leq 7}$ is a basis of $\mathfrak{o}(8)$.
Note that $ G_{ij} = [G_{ik}, G_{kj}] \ (0\leq k\leq 7, k\neq i, j)$.
We can see that $\{F_{ij}\}_{0\leq i < j\leq 7}$ is also a basis of $\mathfrak{o}(8)$ by the following lemma.
\begin{lemma}[ \cite{Yokota}, Lemma~1.3.1]\label{Lem:yokota-basis}
For $i, j = 0,1,\dots,7, i \neq j$, we have $F_{ij} \in \mathfrak{o}(8)$, and when $i < j$,
$F_{ij}$ is expressed in terms of $G_{ij}$ as follows.
\begin{align*}
&\begin{cases} 2F_{01} = G_{01}+G_{23}+G_{45}+G_{67} \\ 2F_{23} = G_{01}+G_{23}-G_{45}-G_{67} \\ 2F_{45} = G_{01}-G_{23}+G_{45}-G_{67} \\ 2F_{67} = G_{01}-G_{23}-G_{45}+G_{67}, \end{cases} &
&\begin{cases} 2F_{02} = G_{02}-G_{13}-G_{46}+G_{57} \\ 2F_{13} = -G_{02}+G_{13}-G_{46}+G_{57} \\ 2F_{46} = -G_{02}-G_{13}+G_{46}+G_{57} \\ 2F_{57} = G_{02}+G_{13}+G_{46}+G_{57}, \end{cases} \\
&\begin{cases} 2F_{03} = G_{03}+G_{12}+G_{47}+G_{56} \\ 2F_{12} = G_{03}+G_{12}-G_{47}-G_{56} \\ 2F_{47} = G_{03}-G_{12}+G_{47}-G_{56} \\ 2F_{56} = G_{03}-G_{12}-G_{47}+G_{56}, \end{cases} &
&\begin{cases} 2F_{04} = G_{04}-G_{15}+G_{26}-G_{37} \\ 2F_{15} = -G_{04}+G_{15}+G_{26}-G_{37} \\ 2F_{26} = G_{04}+G_{15}+G_{26}+G_{37} \\ 2F_{37} = -G_{04}-G_{15}+G_{26}+G_{37}, \end{cases} \\
&\begin{cases} 2F_{05} = G_{05}+G_{14}-G_{27}-G_{36} \\ 2F_{14} = G_{05}+G_{14}+G_{27}+G_{36} \\ 2F_{27} = -G_{05}+G_{14}+G_{27}-G_{36} \\ 2F_{36} = -G_{05}+G_{14}-G_{27}+G_{36}, \end{cases} &
&\begin{cases} 2F_{06} = G_{06}-G_{17}-G_{24}+G_{35} \\ 2F_{17} = -G_{06}+G_{17}-G_{24}+G_{35} \\ 2F_{24} = -G_{06}-G_{17}+G_{24}+G_{35} \\ 2F_{35} = G_{06}+G_{17}+G_{24}+G_{35}, \end{cases} \\
&\begin{cases} 2F_{07} = G_{07}+G_{16}+G_{25}+G_{34} \\ 2F_{16} = G_{07}+G_{16}-G_{25}-G_{34} \\ 2F_{25} = G_{07}-G_{16}+G_{25}-G_{34} \\ 2F_{34} = G_{07}-G_{16}-G_{25}+G_{34}. \end{cases}
\end{align*}
In particular, $\{F_{ij} | 0 \le i < j \le 7\}$ forms a basis of $\mathfrak{o}(8)$.
\end{lemma}
\begin{proposition}[\cite{Yokota}, Lemma~1.3.2, Theorem~1.3.5, \cite{Loos2}, pp.154--155]\label{defofkappa}
We define a linear transformation $\pi^Y$ on $\mathfrak{o}(8)$ by $\pi^Y (G_{ij}) = F_{ij}$. This $\pi^Y$ is an automorphism of $\mathfrak{o}(8)$ with $(\pi^Y)^2=1$. 
The group generated by $\pi^{L}$ and $\pi^Y $ isomorphic to symmetric group $S_{3}$ of degree $3$.
We set $\kappa := \pi^Y \circ \pi^L$, then $\kappa$ is an outer automorphism of order three on $\mathfrak{o}(8)$.
Furthermore, $\pi^{L}, \pi^Y, \sigma^{L}, \kappa$ have the following relations.
$$
(\pi^{L})^2 =1, (\pi^Y)^2=1, (\sigma^{L})^2= 1, \pi^L \sigma^L =\sigma^L \pi^L , \sigma^{L}\kappa= \kappa \sigma^{L}.
$$
\end{proposition}
Based on the above, we define involutions of $\mathfrak{o}(8)$ by
\[ \theta_1 = \pi^L \sigma^L, \quad \theta_2 = \kappa \theta_1 \kappa^{-1}. \]
Then we can see that
\begin{align*} 
&\theta_{1}\theta_{2}=\pi^L \sigma^L \kappa \pi^L \sigma^L \kappa^{-1}
=\pi^L  \kappa \pi^L  \kappa^{-1}
=\pi^L  \pi^{Y} \pi^L \pi^L  \pi^L \pi^Y \\
=&\pi^L  \pi^{Y} \pi^L \pi^Y 
=\kappa^{-2}=\kappa.
\end{align*} 
Thus, the order of $\theta_{1}\theta_{2}$ is three.
By definition, we have 
\begin{align*}
\mathfrak{k}_1
=&\left\{\left.
\begin{pmatrix}
0&      & X       \\
     & Y&        \\
{-}^t X   &      & Z     \\
\end{pmatrix}\right|
\begin{array}{l}
X \in M(1, 4;\mathbb{R}),\\
Y\in \mathfrak{o}(3),\\
Z\in \mathfrak{o}(4)
\end{array}
\right\} \\
=&\mathbb{R} G_{04} \oplus \mathbb{R} G_{05} \oplus \mathbb{R} G_{06} \oplus \mathbb{R} G_{07} \\
&\oplus \mathbb{R} G_{12} \oplus \mathbb{R} G_{13} \oplus \mathbb{R} G_{23} \oplus \mathbb{R} G_{45} \oplus \mathbb{R} G_{46} \oplus \mathbb{R} G_{47} \\ 
&\oplus \mathbb{R} G_{56} \oplus \mathbb{R} G_{57} \oplus \mathbb{R} G_{67}, \\ 
\mathfrak{p}_1
=&\left\{\left.
\begin{pmatrix}
&   X  &        \\
{-}^t X  & &Y        \\
   &    {-}^tY  &     \\
\end{pmatrix}\right|
\begin{array}{l}
X \in M(1, 3;\mathbb{R}),\\
Y \in M(3, 4;\mathbb{R})
\end{array}
\right\} \\
=&\mathbb{R} G_{01} \oplus \mathbb{R} G_{02} \oplus \mathbb{R} G_{03}  \\
&\oplus \mathbb{R} G_{14} \oplus \mathbb{R} G_{15} \oplus \mathbb{R} G_{16} \oplus \mathbb{R} G_{17} \\
&\oplus \mathbb{R} G_{24} \oplus \mathbb{R} G_{25} \oplus \mathbb{R} G_{26} \oplus \mathbb{R} G_{27} \\
&\oplus \mathbb{R} G_{34} \oplus \mathbb{R} G_{35} \oplus \mathbb{R} G_{36} \oplus \mathbb{R} G_{37} .
\end{align*}
Since
$\mathfrak{k}_2=\kappa (\mathfrak{k}_1), \mathfrak{p}_2=\kappa (\mathfrak{p}_1)$
and
\[ \kappa(G_{0i}) = F_{0i} \quad (1 \le i \le 7),\quad   \kappa(G_{ij}) = -F_{ij} \quad (1 \le i,j \le 7),  \]
we have
\begin{align*}
\mathfrak{k}_2
=&\mathbb{R} F_{04} \oplus \mathbb{R} F_{05} \oplus \mathbb{R} F_{06} \oplus \mathbb{R} F_{07} \\
&\oplus \mathbb{R} F_{12} \oplus \mathbb{R} F_{13} \oplus \mathbb{R} F_{23} \oplus \mathbb{R} F_{45} \oplus \mathbb{R} F_{46} \oplus \mathbb{R} F_{47} \\ 
&\oplus \mathbb{R} F_{56} \oplus \mathbb{R} F_{57} \oplus \mathbb{R} F_{67}, \\
\mathfrak{p}_2
=&\mathbb{R} F_{01} \oplus \mathbb{R} F_{02} \oplus \mathbb{R} F_{03}  \\
&\oplus \mathbb{R} F_{14} \oplus \mathbb{R} F_{15} \oplus \mathbb{R} F_{16} \oplus \mathbb{R} F_{17} \\
&\oplus \mathbb{R} F_{24} \oplus \mathbb{R} F_{25} \oplus \mathbb{R} F_{26} \oplus \mathbb{R} F_{27} \\
&\oplus \mathbb{R} F_{34} \oplus \mathbb{R} F_{35} \oplus \mathbb{R} F_{36} \oplus \mathbb{R} F_{37}. 
\end{align*}
Hence we get two canonical decompositions 
$\mathfrak{g}
=\mathfrak{k}_1 \oplus \mathfrak{p}_1
=\mathfrak{k}_2 \oplus \mathfrak{p}_2.$

Moreover, using Lemma~\ref{Lem:yokota-basis} shows
\begin{align*}
\mathfrak{k}_1 \cap \mathfrak{k}_2 
=& \{ aG_{23} + bG_{45} + cG_{67} | a+b+c=0 \} \\
&\oplus \{ aG_{13} + bG_{46} + cG_{57} | -a-b+c=0 \} \\
&\oplus \{ aG_{12} + bG_{47} + cG_{56} | a+b+c=0 \} \\
=& \{ aF_{23} + bF_{45} + cF_{67} | a+b+c=0 \} \\
&\oplus \{ aF_{13} + bF_{46} + cF_{57} | -a-b+c=0 \} \\
&\oplus \{ aF_{12} + bF_{47} + cF_{56} | a+b+c=0 \} \\
\end{align*}
and
\begin{align*}
\mathfrak{p}_1 \cap \mathfrak{p}_2 
=& \{ aG_{15} + bG_{26} + cG_{37} | -a+b-c=0 \} \\
&\oplus \{ aG_{14} + bG_{27} + cG_{36} | a-b-c=0 \} \\
&\oplus \{ aG_{17} + bG_{24} + cG_{35} | -a-b+c=0 \} \\
&\oplus \{ aG_{16} + bG_{25} + cG_{34} | a+b+c=0 \}\\
=& \{ aF_{15} + bF_{26} + cF_{37} | -a+b-c=0 \} \\
&\oplus \{ aF_{14} + bF_{27} + cF_{36} | a-b-c=0 \} \\
&\oplus \{ aF_{17} + bF_{24} + cF_{35} | -a-b+c=0 \} \\
&\oplus \{ aF_{16} + bF_{25} + cF_{34} | a+b+c=0 \}.
\end{align*}
\subsubsection*{Maximal abelian subspaces and root space decompositions
}
In this subsection, we verify that the compact symmetric triad $(\mathfrak{g}, \theta_{1}, \theta_{2})$ is canonical.
It is sufficient to find a maximal abelian subalgebra \(\mathfrak{t}\) satisfying the conditions in Definition~\ref{dfn:canonical}. 
Let $\mathfrak{t} = \mathrm{span}\{ G_{04}, G_{15}, G_{26}, G_{37} \}$. 
This is a maximal abelian subalgebra of $\mathfrak{g}=\mathfrak{o}(8)$. 
From Lemma~\ref{Lem:yokota-basis}, 
we have $\mathfrak{t} =\mathrm{span}\{ F_{04}, F_{15}, F_{26}, F_{37} \}$ and 
$\mathfrak{t}$ is $\theta_{i}$-invariant ($i=1, 2$).
We also have 
$\mathfrak{a}_1 :=\mathfrak{t}\cap {\mathfrak p}_1= \mathrm{span}\{ G_{15}, G_{26}, G_{37} \}$ and 
$\mathfrak{a}_2 := \mathfrak{t}\cap {\mathfrak p}_2= \mathrm{span}\{ F_{15}, F_{26}, F_{37} \}$ 
are maximal abelian subspaces of 
$\mathfrak{p}_{1}$ and $\mathfrak{p}_{2}$, respectively.
We have $\mathfrak{a}:= \mathfrak{a}_1 \cap \mathfrak{a}_2=\{ aG_{15} + bG_{26} + cG_{37} | -a+b-c=0 \} =\{ aF_{15} + bF_{26} + cF_{37} | -a+b-c=0 \}.$ 
We can see that $\mathfrak{t} = \mathfrak{a}_1 + \mathfrak{a}_2$ and 
$\mathfrak{a}$ is a maximal abelian subspace of $\mathfrak{p}_{1} \cap \mathfrak{p}_{2}$.

Now, we consider the root space decomposition of $\mathfrak{o}(8)$ with respect to $\mathfrak{t}$.
For an invariant inner product on \(\mathfrak{g}\) such that \(\{ G_{ij} \mid 0\leq i<j\leq 7\}\) is an orthonormal basis, the set of roots is given by
\[ \{ \pm G_{i,i+4} \pm G_{j, j+4} | 0 \le i < j \le 3 \} 
= \{ \pm F_{i,i+4} \pm F_{j, j+4} | 0 \le i < j \le 3 \} \cong D_{4}.\]
Let
\begin{align*}
\alpha_1&=G_{04}-G_{15}=F_{04}-F_{15}, \\
\alpha_2&=G_{15}+G_{26}=F_{15}+F_{26} , \\
\alpha_3&=-G_{26}+G_{37}= -F_{04}-F_{15}, \\
\alpha_4&=-G_{04}-G_{15}=-F_{26}+F_{37}.\\
\end{align*}
Then, $\Pi =\{\alpha_1, \alpha_2, \alpha_3, \alpha_4\}$ is a fundamental system of $D_{4}$.
Moreover, from this expression, we can verify that
\begin{align}
\theta_1 (\alpha_1)=&-\alpha_4, \;\theta_1 (\alpha_2)=-\alpha_2, \;\theta_1 (\alpha_3)=-\alpha_3, \; \theta_1 (\alpha_4)=-\alpha_1,  \label{simpleroots}  \\
\theta_2 (\alpha_1)=&-\alpha_3, \;\theta_2 (\alpha_2)=-\alpha_2,\; \theta_2 (\alpha_3)=-\alpha_1, \; \theta_2 (\alpha_4)=-\alpha_4. \nonumber
\end{align}
Thus, we see that the order of $\theta_{1}\theta_{2}|_{\mathfrak{t}}$ is three.
The set of positive roots \(\Delta^{+}\) with respect to this fundamental system is given by
\begin{align*}
\Delta^{+}=\{ &\alpha_1, \alpha_2, \alpha_3, \alpha_4, \alpha_1+\alpha_2, \alpha_1+\alpha_2+\alpha_3, \alpha_1+\alpha_2+\alpha_3+\alpha_4, \\
&\alpha_1+2\alpha_2+\alpha_3+\alpha_4, \alpha_2+\alpha_3+\alpha_4, \alpha_2+\alpha_3, \alpha_2+\alpha_4, \alpha_1+\alpha_2+\alpha_4 \}\cong D_{4}^{+}. 
\end{align*}

The spin group $Spin(8)$ has $\mathfrak{o}(8)$ as its Lie algebra. 
Since $Spin(8)$ is simply connected, any automorphism of $\mathfrak{o}(8)$ lifts to an automorphism of 
$Spin(8)$. 
Denoting the lifted automorphism by the same symbol, 
$(Spin(8), \theta_1 , \theta_2)$ is a compact symmetric triad.
Summarizing the above, we obtain the following proposition.

\begin{proposition}
The compact symmetric triad $(Spin(8),\theta_1,\theta_2)$ is canonical and
the above maximal abelian subalgebra $\mathfrak{t}$ of $\mathfrak{o}(8)$ is canonical with respect to $(Spin(8),\theta_1,\theta_2)$. 
\end{proposition}

In the following, we explicitly decompose the Lie algebra $\mathfrak{o}(8)$ to determine symmetric triad and their multiplicities corresponding to $(Spin(8),\theta_1,\theta_2)$. 
For each positive root $\alpha$, there exist unit vectors $S_{\alpha}, T_{\alpha}$ in \(\mathfrak{o}(8)\) such that
\[
[H, S_{\alpha}]= \langle \alpha , H \rangle T_{\alpha},\quad 
[H, T_{\alpha}]= - \langle \alpha , H \rangle S_{\alpha},\quad 
[S_{\alpha}, T_{\alpha}]=\langle \alpha , \alpha \rangle \alpha .
\]
For instance, $S_{\alpha}$ and $T_{\alpha}$ can be chosen as in the following table.

\begin{center}
\begin{tabular}{| l|l| l |l}\hline
$\alpha$ & $\sqrt{2} S_{\alpha}$& $\sqrt{2} T_{\alpha}$\\ \hline
$\alpha_1$ & $G_{01}+G_{45}$ &$G_{05}+G_{14}$ \\
$\alpha_2$ & $G_{16}-G_{25}$&$ G_{12}-G_{56}$ \\
$\alpha_3$ & $-G_{23}-G_{67}$&$G_{27}+G_{36}$ \\
$\alpha_4$ & $-G_{01}+G_{45}$&$-G_{05}+G_{14}$ \\
$\alpha_1+\alpha_2$ & $G_{06}-G_{24} $&$G_{02}-G_{46}$  \\
$\alpha_2+\alpha_3$ & $G_{17}-G_{35}$&$G_{13}-G_{57}$ \\
$\alpha_2+\alpha_4$ & $-G_{06}-G_{24} $&$-G_{02}-G_{46}$ \\
$\alpha_1+\alpha_2+\alpha_3$ & $G_{07}-G_{34}$&$G_{03}-G_{47}$ \\
$\alpha_1+\alpha_2+\alpha_4$ & $G_{16}+G_{25}$&$G_{12}+G_{56}  $\\
$\alpha_2+\alpha_3+\alpha_4$ & $-G_{07}-G_{34}$&$-G_{03}-G_{47} $\\
$\alpha_1+\alpha_2+\alpha_3+\alpha_4$ & $G_{13}+G_{57}$&$ -G_{17}-G_{35}$\\
$\alpha_1+2\alpha_2+\alpha_3+\alpha_4$ & $G_{27}-G_{36}$&$G_{23}-G_{67}$\\ \hline
\end{tabular}
\end{center}
From this table and the relations \( \kappa(G_{0i}) = F_{0i} \quad (1 \le i \le 7),\quad   \kappa(G_{ij}) = -F_{ij} \quad (1 \le i,j \le 7) \), it follows that
$\kappa(\alpha_1)=\alpha_3,  \kappa(\alpha_3)=\alpha_4, \kappa(\alpha_4)=\alpha_1$, and
$\kappa(\alpha_2) =\alpha_2$.
Furthermore, we can see that
\begin{align*}
S_{\kappa (\alpha)} =\kappa (S_{\alpha}),& \quad T_{\kappa (\alpha)} =\kappa (T_{\alpha})\\
S_{-\theta_{i} (\alpha)} = \theta_{i} (S_{\alpha}),& \quad T_{-\theta_{i} (\alpha)} =- \theta_{i} (T_{\alpha}).
\end{align*}
Since the orthogonal projection from \(\mathfrak{t}\) to \(\mathfrak{a}\) is given by \(X \mapsto \frac{X+ \kappa(X)+ \kappa^2(X)}{3}\), 
we have
\[
\tilde{\Sigma}^{+}=\{\lambda_{1}, \lambda_{2}, \lambda_{1}+\lambda_{2}, 2\lambda_{1}+\lambda_{2}, 3\lambda_{1}+\lambda_{2}, 3\lambda_{1}+2\lambda_{2}\}\cong{G_{2}}^{+}
\]
where
\(
\lambda_{1}=\frac{\alpha_{1}+\alpha_{3}+\alpha_{4} }{3}, 
\lambda_{2}=\alpha_{2}.
\)

Then \(\mathfrak{g}^{\mathbb{C} }\) and \(\mathfrak{g}\) decomposes as follows.
\begin{align*}
&\mathfrak{g}^{\mathbb{C} }
=\mathfrak{g}(0) \oplus \sum_{\lambda \in \tilde{\Sigma}^{+}} (\mathfrak{g}(\mathfrak{a}, \lambda) \oplus \mathfrak{g}(\mathfrak{a}, -\lambda)),\\
&\mathfrak{g}
=(\mathfrak{g}(0) \cap \mathfrak{g})\oplus
 \sum_{\lambda \in \tilde{\Sigma}^{+}}
( (\mathfrak{g}(\mathfrak{a}, \lambda) \oplus \mathfrak{g}(\mathfrak{a}, -\lambda))\cap \mathfrak{g}).
\end{align*}
In this case, $\mathfrak{g}(0)=\mathfrak{t}^{\mathbb{C}}$ holds.
Then we can see that 
\begin{align*}
((\mathfrak{g}(\mathfrak{a}, \lambda) \oplus \mathfrak{g}(\mathfrak{a}, -\lambda)))\cap\mathfrak{g}
=&\{X \in \mathfrak{g} \mid [H, [H, X]] =-\langle H, \lambda \rangle^{2} X \ (H\in \mathfrak{a})\} \\
=&\mathrm{span}\{S_{\alpha} , T_{\alpha } \mid   \lambda =\frac{\alpha + \kappa(\alpha)+ \kappa^2(\alpha)}{3}  \},
\end{align*}

\begin{align*}
&(\mathfrak{g}(\mathfrak{a}, \lambda_1) \oplus \mathfrak{g}(\mathfrak{a}, -\lambda_1) )\cap\mathfrak{g}
= \mathrm{span}
\begin{Bmatrix}
    G_{01}+G_{45}, & -G_{23}-G_{67}, & -G_{01}+G_{45} \\
    G_{05}+G_{14}, &  G_{27}+G_{36}, & -G_{05}+G_{14}
\end{Bmatrix},
\\[2em]
&(\mathfrak{g}(\mathfrak{a},\lambda_2) \oplus \mathfrak{g}(\mathfrak{a}, -\lambda_2) )\cap\mathfrak{g}
= \mathrm{span}
\begin{Bmatrix}
    G_{16}-G_{25} \\
    G_{12}-G_{56}
\end{Bmatrix},
\\[2em]
&(\mathfrak{g}(\mathfrak{a}, \lambda_1+\lambda_2) \oplus \mathfrak{g}(\mathfrak{a}, -(\lambda_1+\lambda_2)))\cap\mathfrak{g} \\
&\qquad = \mathrm{span}
\begin{Bmatrix}
    G_{06}-G_{24}, & G_{17}-G_{35}, & -G_{06}-G_{24} \\
    G_{02}-G_{46}, & G_{13}-G_{57}, & -G_{02}-G_{46}
\end{Bmatrix},
\\[2em]
&(\mathfrak{g}(\mathfrak{a}, 2\lambda_1+\lambda_2) \oplus \mathfrak{g}(\mathfrak{a}, -(2\lambda_1+\lambda_2)))\cap\mathfrak{g}\\
&\qquad = \mathrm{span}
\begin{Bmatrix}
    G_{07}-G_{34}, &  G_{16}+G_{25}, & -G_{07}-G_{34} \\
    G_{03}-G_{47}, &  G_{12}+G_{56}, & -G_{03}-G_{47}
\end{Bmatrix},
\\[2em]
&(\mathfrak{g}(\mathfrak{a}, 3\lambda_1+\lambda_2) \oplus \mathfrak{g}(\mathfrak{a}, -(3\lambda_1+\lambda_2)))\cap\mathfrak{g} 
= \mathrm{span}
\begin{Bmatrix}
     G_{13}+G_{57} \\
    -G_{17}-G_{35}
\end{Bmatrix},
\\[2em]
&(\mathfrak{g}(\mathfrak{a}, 3\lambda_1+2\lambda_2) \oplus \mathfrak{g}(\mathfrak{a}, -(3\lambda_1+2\lambda_2)))\cap\mathfrak{g} 
= \mathrm{span}
\begin{Bmatrix}
    G_{27}-G_{36} \\
    G_{23}-G_{67}
\end{Bmatrix}.
\end{align*}
For the long roots \(\lambda = \lambda_{2}, 3\lambda_{1}+\lambda_{2}, 3\lambda_{1}+2\lambda_{2}\), we have \(\lambda \in \Delta\).
Since
\(\mathfrak{g}(\mathfrak{a},\lambda) \oplus \mathfrak{g}(\mathfrak{a}, -\lambda) 
=
 \mathrm{span}\{S_{\lambda }, T_{\lambda }\}
\)
and \( \kappa (S_{\lambda })=S_{\lambda }, \kappa (T_{\lambda })=T_{\lambda }\), 
we have
\(\mathfrak{g}(\mathfrak{a},\lambda) \oplus \mathfrak{g}(\mathfrak{a}, -\lambda) = \mathfrak{g}(\mathfrak{a},\lambda, 1) \oplus \mathfrak{g}(\mathfrak{a}, -\lambda, 1)\). 
This means that 
\(\mathfrak{g}(\mathfrak{a},\pm \lambda,\omega) = \mathfrak{g}(\mathfrak{a}, \pm \lambda,\overline{\omega})=\{0\}. \)
Then 
we obtain
\[
\mathfrak{g}(\mathfrak{a},\lambda,1)=\mathbb{C}(S_{\lambda }-\sqrt{-1} T_{\lambda }),\qquad 
\mathfrak{g}(\mathfrak{a},-\lambda, 1)=\mathbb{C}(S_{\lambda }+\sqrt{-1} T_{\lambda }).
\]
For the short roots \(\lambda = \lambda_{1}, \lambda_{1}+\lambda_{2}, 2\lambda_{1}+\lambda_{2}\), if we choose an \(\alpha \in \Delta^{+}\) such that 
\(
\lambda=\frac{(\alpha+ \kappa(\alpha)+ \kappa^2(\alpha))}{3}\),
we have
\[\mathfrak{g}(\mathfrak{a},\lambda) \oplus \mathfrak{g}(\mathfrak{a}, -\lambda) 
=
 \mathrm{span}\{S_{\alpha }, S_{\kappa( \alpha) }, S_{\kappa^{2}(\alpha) }, T_{\alpha }, T_{\kappa( \alpha) }, T_{\kappa^{2}(\alpha) } \}.
\]
Thus, by setting
\begin{align*}
X_{\lambda}^{1}=\frac{S_{\alpha }+ S_{\kappa( \alpha) }+ S_{\kappa^{2}(\alpha)}}{\sqrt{3}},\quad 
Y_{\lambda}^{1}=\frac{T_{\alpha }+ T_{\kappa( \alpha) }+ T_{\kappa^{2}(\alpha)}}{\sqrt{3}},
\end{align*}
we can see that 
$X_{\lambda}^{1}-\sqrt{-1} Y_{\lambda}^{1} \in \mathfrak{g}(\mathfrak{a},\lambda,1)$
and
$X_{\lambda}^{1}+\sqrt{-1} Y_{\lambda}^{1} \in \mathfrak{g}(\mathfrak{a},-\lambda, 1)$. 

Furthermore, if we take \(\alpha\) satisfying \(\theta_{1}(\alpha) =-\alpha\) and set
\begin{align*}
\sqrt{2}X_{\lambda}^{\omega}= \frac{S_{\kappa (\alpha) }+ S_{\kappa^2( \alpha) }- 2 S_{\alpha}}{\sqrt{6}} +\frac{T_{\kappa (\alpha) }- T_{\kappa^2( \alpha)}}{\sqrt{2}},\\ 
\sqrt{2}Y_{\lambda}^{\omega}=-\frac{S_{\kappa (\alpha) }- S_{\kappa^2( \alpha)}}{\sqrt{2}}+\frac{T_{\kappa (\alpha) }+T_{\kappa^2( \alpha) }- 2 T_{\alpha}}{\sqrt{6}},\\
\sqrt{2}X_{\lambda}^{\bar{\omega}}=\frac{S_{\kappa (\alpha) }+ S_{\kappa^2( \alpha) }- 2 S_{\alpha}}{\sqrt{6}} -\frac{T_{\kappa (\alpha) }- T_{\kappa^2( \alpha)}}{\sqrt{2}},\\
\sqrt{2}Y_{\lambda}^{\bar{\omega}}=\frac{S_{\kappa (\alpha) }- S_{\kappa^2( \alpha)}}{\sqrt{2}}+\frac{T_{\kappa (\alpha) }+ T_{\kappa^2( \alpha) }- 2 T_{\alpha}}{\sqrt{6}},\end{align*}
then we have
\(X_{\lambda}^{\omega}, X_{\lambda}^{\bar{\omega}}  \in \mathfrak{k}_{1} \) and
\(Y_{\lambda}^{\omega}, Y_{\lambda}^{\bar{\omega}}  \in \mathfrak{p}_{1} \).

In this case,
\begin{align*}
\kappa (X_{\lambda}^{\omega})
=&\frac{S_{\kappa^2 (\alpha) }+ S_{ \alpha }- 2 S_{\kappa(\alpha)}}{2\sqrt{3}} +\frac{T_{\kappa^2 (\alpha) }- T_{ \alpha}}{2}\\ 
=&-\frac{1}{2} \frac{1}{2\sqrt{3}}(S_{\kappa (\alpha) }+ S_{\kappa^2( \alpha) }- 2 S_{\alpha})
-\frac{3}{2} \frac{1}{2\sqrt{3}}(S_{\kappa (\alpha) }- S_{\kappa^2( \alpha)})\\
&-\frac{1}{2}\frac{1}{2}(T_{\kappa (\alpha) }- T_{\kappa^2( \alpha)})
+\frac{1}{2}\frac{1}{2}(T_{\kappa (\alpha) }+T_{\kappa^2( \alpha) }- 2 T_{\alpha}) \\
=&-\frac{1}{2}X_{\lambda}^{\omega} +\frac{\sqrt{3}}{2}Y_{\lambda}^{\omega}
=\mathrm{re} (\omega) X_{\lambda}^{\omega} +\mathrm{im} (\omega)Y_{\lambda}^{\omega}, \\
\kappa (Y_{\lambda}^{\omega})
=&-\frac{S_{\kappa^2 (\alpha) }- S_{ \alpha}}{2}+\frac{T_{\kappa^2 (\alpha) }+T_{ \alpha }- 2 T_{\kappa(\alpha)}}{2\sqrt{3}}\\
=&-\frac{1}{2} \frac{1}{2}(S_{\kappa (\alpha) }+ S_{\kappa^2( \alpha) }- 2 S_{\alpha})
+\frac{1}{2} \frac{1}{2}(S_{\kappa (\alpha) }- S_{\kappa^2( \alpha)})\\
&-\frac{3}{2}\frac{1}{2\sqrt{3}}(T_{\kappa (\alpha) }- T_{\kappa^2( \alpha)})
-\frac{1}{2}\frac{1}{2\sqrt{3}}(T_{\kappa (\alpha) }+T_{\kappa^2( \alpha) }- 2 T_{\alpha}) \\
=&-\frac{\sqrt{3}}{2}X_{\lambda}^{\omega} -\frac{1}{2}Y_{\lambda}^{\omega}
=-\mathrm{im}( \omega )X_{\lambda}^{\omega} +\mathrm{re}(\omega )Y_{\lambda}^{\omega},
\end{align*}
\begin{align*}
\kappa (X_{\lambda}^{\bar{\omega}})
=&\frac{S_{\kappa^2 (\alpha) }+ S_{ \alpha }- 2 S_{\kappa(\alpha)}}{2\sqrt{3}} -\frac{T_{\kappa^2 (\alpha) }- T_{ \alpha}}{2}\\ 
=&-\frac{1}{2} \frac{1}{2\sqrt{3}}(S_{\kappa (\alpha) }+ S_{\kappa^2( \alpha) }- 2 S_{\alpha})
-\frac{3}{2} \frac{1}{2\sqrt{3}}(S_{\kappa (\alpha) }- S_{\kappa^2( \alpha)})\\
&+\frac{1}{2}\frac{1}{2}(T_{\kappa (\alpha) }- T_{\kappa^2( \alpha)})
-\frac{1}{2}\frac{1}{2}(T_{\kappa (\alpha) }+T_{\kappa^2( \alpha) }- 2 T_{\alpha}) \\
=&-\frac{1}{2}X_{\lambda}^{\bar{\omega}} -\frac{\sqrt{3}}{2}Y_{\lambda}^{\bar{\omega}}
= \mathrm{re} (\bar{\omega})X_{\lambda}^{\bar{\omega}} +\mathrm{im} (\bar{\omega})Y_{\lambda}^{\bar{\omega}}, \\
\kappa (Y_{\lambda}^{\bar{\omega}})
=&\frac{S_{\kappa^2 (\alpha) }- S_{ \alpha}}{2}+\frac{T_{\kappa^2 (\alpha) }+T_{ \alpha }- 2 T_{\kappa(\alpha)}}{2\sqrt{3}}\\
=&\frac{1}{2} \frac{1}{2}(S_{\kappa (\alpha) }+ S_{\kappa^2( \alpha) }- 2 S_{\alpha})
-\frac{1}{2} \frac{1}{2}(S_{\kappa (\alpha) }- S_{\kappa^2( \alpha)})\\%
&-\frac{3}{2}\frac{1}{2\sqrt{3}}(T_{\kappa (\alpha) }- T_{\kappa^2( \alpha)})
-\frac{1}{2}\frac{1}{2\sqrt{3}}(T_{\kappa (\alpha) }+T_{\kappa^2( \alpha) }- 2 T_{\alpha}) \\
=&\frac{\sqrt{3}}{2}X_{\lambda}^{\bar{\omega}} -\frac{1}{2}Y_{\lambda}^{\bar{\omega}}
= -\mathrm{im}( \bar{\omega} ) X_{\lambda}^{\bar{\omega}} +\mathrm{re}(\bar{\omega} ) Y_{\lambda}^{\bar{\omega}}.
\end{align*}
From the above calculations, we obtain
\begin{align*}
\mathfrak{g}(\mathfrak{a},\lambda,\omega )=\mathbb{C}(X_{\lambda}^{\omega}-\sqrt{-1} Y_{\lambda}^{\omega}),\qquad 
\mathfrak{g}(\mathfrak{a},-\lambda, \bar{\omega})=\mathbb{C}(X_{\lambda}^{\omega}+\sqrt{-1} Y_{\lambda}^{\omega}), \\
\mathfrak{g}(\mathfrak{a},\lambda,\bar{\omega} )=\mathbb{C}(X_{\lambda}^{\bar{\omega}}-\sqrt{-1} Y_{\lambda}^{\bar{\omega}}),\qquad 
\mathfrak{g}(\mathfrak{a},-\lambda, \omega)=\mathbb{C}(X_{\lambda}^{\bar{\omega}}+\sqrt{-1} Y_{\lambda}^{\bar{\omega}}).
\end{align*}
Therefore,
\begin{align*}
\Sigma_{1}=&G_{2}
=\{\pm \lambda_{1}, \pm \lambda_{2}, \pm (\lambda_{1}+\lambda_{2}), \pm (2\lambda_{1}+\lambda_{2}), \pm (3\lambda_{1}+\lambda_{2}), \pm (3\lambda_{1}+2\lambda_{2})\}, \\
\Sigma_{\omega }=&\Sigma_{\bar{\omega}}
=\{\text{short roots in } G_{2}\}=A_{2}
=\{\pm \lambda_{1}, \pm (\lambda_{1}+\lambda_{2}), \pm (2\lambda_{1}+\lambda_{2})\}
\end{align*}
and all multiplicities are \(1\).
Thus, letting \(H_{1}, H_{2}\) be the dual basis of \(\lambda_{1}, \lambda_{2}\) with respect to \(\langle , \rangle \), 
we have
\[
H_1=-G_{15}+G_{26}+2G_{37}=6 \lambda_{1}+3\lambda_{2}, 
\quad 
H_2=G_{26}+G_{37}=3 \lambda_{1}+2\lambda_{2}.
\]
Then we obtain
\begin{align}\label{eq:regularpointsord3}
\mathfrak{a}_{r}=&\{ H \in \mathfrak{a} \mid \langle \lambda , H \rangle \not\in \pi\mathbb{Z} \ (\lambda \in \Sigma_{1})  \}
\cap \{ H \in \mathfrak{a} \mid \langle \lambda , H \rangle \not\in  \pm \dfrac{\pi}{3} +\pi\mathbb{Z} \ (\lambda \in \Sigma_{\omega})  \}  \\
=&\left\{
\dfrac{\pi}{3}(x_1H_1+x_2H_2 ) \left|
\begin{array}{l}
x_1,x_2,x_1+x_2,2x_1+x_2,3x_1+x_2,3x_1+2x_2\not\in 3\mathbb{Z},\\
x_1\pm 1  ,x_1+x_2\pm1 ,2x_1+x_2\pm1 
\not\in 3\mathbb{Z}
\end{array}
\right.
\right\}\nonumber \\
=&\left\{
\dfrac{\pi}{3} (x_1H_1+x_2H_2)\left|
\begin{array}{l}
x_1,x_1+x_2,2x_1+x_2\not\in \mathbb{Z},\\
x_2,3x_1+x_2,3x_1+2x_2\not\in 3 \mathbb{Z}
\end{array}
\right.
\right\}, \nonumber
\end{align}
which is invariant under the multiplication by $-1$.

Thus, a cell $P$ is given by
\[
P=\left\{
\frac{\pi}{3}(x_1H_1+x_2H_2)\mid x_i>0\; (i=1,2), 2x_1+x_2<1
\right\}.
\]
Hence, the closure of $P$ is a simplex. 

Then we give an example of canonical form whose $\mathfrak{a}_{r}$ is not invariant under the multiplication by $-1$.

\begin{example}\label{ex:multiplicationby-1}\rm 
Take $(G,\theta_1,\theta_2)$ as in this section. 
We retain the notation in this section. 
Set $H:= \frac{\pi}{3}H_2\in\mathfrak{a}$. 
Then, for each $\lambda\in \tilde{\Sigma}$, we have $\langle\lambda,H\rangle \in\frac{\pi}{3}\mathbb{Z}$. 
Define two involutions $\theta_1'$ and $\theta_2'$ on $G$ by $\theta_1'=\tau_{\exp H}\theta_1 \tau_{\exp H}^{-1}$ and $\theta_2'=\theta_2$, 
where $\tau_{\exp H}$ is an inner automorphism on $G$ defined by $\exp H$. 
Then $(G,\theta_1',\theta_2')$ is also canonical, and 
\[
\Sigma_1'=A_2,\;
\Sigma_{\omega}'
=A_2\cup\{\lambda_2,3\alpha_1+\alpha_2,-(3\alpha_1+2\alpha_2)\},\;
\Sigma_{\bar{\omega}}'=\{-\lambda\mid \lambda\in \Sigma_{\omega}\}. 
\]
The set ${\mathfrak a}_r'$ of regular points of $(G,\theta_1',\theta_2')$ with respect to $\mathfrak{a}$ is given by 
\begin{align*}
{\mathfrak a}_r'
&=\left\{
H\in\mathfrak{a}\left|
\begin{array}{ll}
\langle \alpha,H\rangle \not\in\frac{\pi}{3}\mathbb{Z} & (\alpha\in A_2),\\
\langle \alpha,H\rangle +\frac{\pi}{3}\not\in\mathbb{Z} & (\alpha=\lambda_2,3\lambda_1+\lambda_2,-(3\lambda_1+2\lambda_2)
\end{array}
\right.
\right\}\\
&=\left\{\frac{\pi}{3}(y_1H_1+y_2H_2)\left|
\begin{array}{l}
y_i, y_1+y_2\not\in\mathbb{Z},\\
-y_1+y_2+1\not\in 3\mathbb{Z},\\
2y_1+y_2+1\not\in 3\mathbb{Z},\\
-y_1-2y_2+1\not\in 3\mathbb{Z}
\end{array}
\right.
\right\}.
\end{align*}
Put 
$X=\frac{\pi}{3}(-tH_1-\frac{1-t}{2}H_2)$. 
Then $-X\not\in\mathfrak{a}_r'$ for any $t$. 
When $t=1/2$, we have $X\in {\mathfrak a}_r'$. \hfill$\Box$
\end{example}
We describe $\#(W(G_2)x_0)\ (x_0\in\mathfrak{a}-\{0\})$. 
Without loss of generalities, we can assume 
$x_0=tH_1+(1-t)H_2\;(0\leq t\leq 1)$.
Then
\[
\#(W(G_2)x_0)=\left\{
\begin{array}{ll}
6 & (t=0,1),\\
12 & (0<t<1).
\end{array}
\right.
\]

\vspace{-4\baselineskip}

\begin{center}
\includegraphics[width=0.35\textwidth]{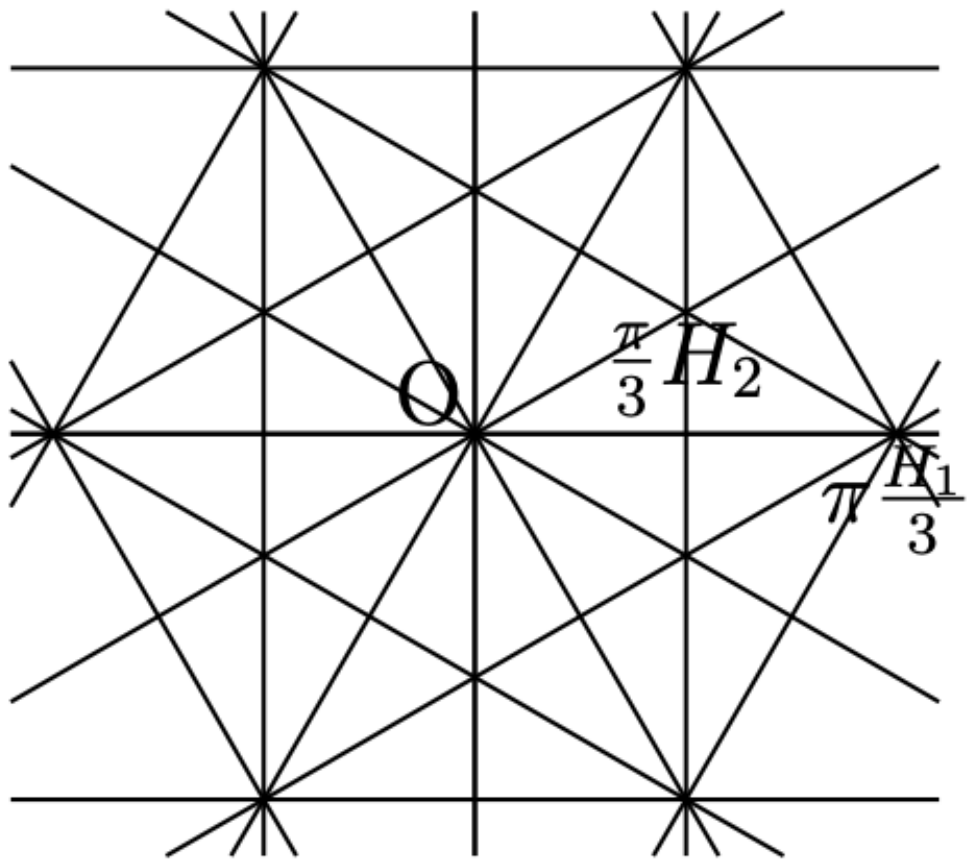}
\end{center}

\vspace{-4\baselineskip}

Recall that $G=Spin(8)$.
As is well known, there exists a covering homomorphism $p:Spin(8)\to SO(8)$, and furthermore, $\mathrm{Ad}(Spin(8))$ and $\mathrm{Ad}(SO(8))$ coincide as automorphism groups of $\mathfrak{o}(8)$. 
Therefore, $SO(8)$ acts transitively on each adjoint orbit. 
Without loss of generality, we can take the base point $x_{0}$ of the orbit from a maximal abelian subspace $\mathfrak{t}$ of $\mathfrak{o}(8)$.
Then we obtain the complex flag manifold $M=\mathrm{Ad}(G)x_0\cong SO(8)/(SO(8))_{x_{0}}$.
The isotropy subgroup $(SO(8))_{x_{0}}$ of the orbit $\mathrm{Ad}(G)x_{0}$ can be described by a painted Dynkin diagram determined by $x_0$ and the root system $\Delta$ (\cite[pp.~98--100]{Arvanitoyeorgos}).
Specifically, this diagram is determined by which roots of $\Delta$ are orthogonal to $x_{0}$.
Here, we assume $x_{0}\in \mathfrak{a}$ to consider the intersection of the real flag manifolds $L_{i}\ (i=1, 2)$ determined by $\theta_{i}\ (i=1, 2)$, 
and the above argument shows that we may take $x_0=tH_1+(1-t)H_2\;(0\leq t\leq 1)$.
The real forms $L_{i}\ (i=1, 2) $ can be expressed as $L_{i}=\mathrm{Ad}(K_{i})x_0\cong p(K_{i})/(p(K_{i})\cap(SO(8))_{x_{0}})$ for each $x_0=tH_1+(1-t)H_2\;(0\leq t\leq 1)$.
By Theorems~\ref{thm:int} and \ref{thm:discrete}, it follows that 
$$
\mathrm{Ad}(\exp H)L_{1}\cap L_{2}=W(G_2)x_0
$$
for $x_0=tH_1+(1-t)H_2\;(0\leq t\leq 1)$ and $H\in \mathfrak{a}_{r}$.

For $0<t<1$, $x_{0}$ is not orthogonal to any root. 
Therefore, $(SO(8))_{x_{0}} $ is isomorphic to $U(1)^4$.
Thus we have $M\cong  SO(8)/U(1)^4$. 
Geometrically, this corresponds to the full flag manifold $F_{2,2,2,2}(\mathbb{R}^{8})$.
Then $\#(\mathrm{Ad}(\exp H)L_{1}\cap L_{2})=12$.

When $t=0$, we have $x_{0}=H_{2}$, and $x_{0}$ is orthogonal to $\alpha_1, \alpha_3, \alpha_4$.
Consequently, $(SO(8))_{x_{0}} $ is isomorphic to $ U(2) \times SO(4)$.
Thus we have $M\cong SO(8)/(U(2) \times SO(4))$. 
Then $\#(\mathrm{Ad}(\exp H)L_{1}\cap L_{2})=6$.

When $t=1$, we have $x_{0}=H_{1}$, and $x_{0}$ is orthogonal to $\alpha_2$.
Consequently, $(SO(8))_{x_{0}} $ is isomorphic to $U(1)^2 \times U(2)$.
Thus we have $M\cong  SO(8)/(U(1)^2 \times U(2))$. 
Then $\#(\mathrm{Ad}(\exp H)L_{1}\cap L_{2})=6$.

Finally, we compare this example with known results. 
From the definition of $\theta_{2}$, we have $K_{2}=\kappa (K_{1})$. Since $\kappa=\theta_{1}\theta_{2}$, it follows that $\kappa (x_{0})=x_{0}$. 
From this, we see that
$$
\kappa (L_{1})=\kappa (\mathrm{Ad}(K_{1})x_0)= \mathrm{Ad}(\kappa(K_{1})) \kappa(x_0)=\mathrm{Ad}(K_{2}) x_0=L_{2}.
$$
Thus, since $L_{1}$ and $L_{2}$ are mapped to each other by the holomorphic isometry $\kappa$ of $M$, they are equivalent as real forms. 
However, the relative position of $L_{1}$ and $L_{2}$ is completely different from that of $L_{1}$ and a real form $L_{1}':=\mathrm{Ad}(g)L_{1}$ for some $g\in G$. 
The case where two real flag manifolds are mapped to each other by the action of $\mathrm{Ad}(G)$ (i.e., they are congruent) has already been investigated in \cite{IIOST}. 
In this case, under our setting, we consider the case of $(G, \theta_{1}, \theta_{1})$, and $\tilde{\Sigma}$ becomes a root system of type $B_{3}$. 
Therefore, when $L_{1}$ and $\mathrm{Ad}(g)L_{1} \ (g\in G)$ intersect discretely, the intersection is obtained as an orbit of the Weyl group of the root system of type $B_{3}$, and the numbers of elements are respectively: 
$\#(L_{1}\cap \mathrm{Ad}(g)L_{1})=48$ for $M\cong  SO(8)/U(1)^4$, 
$\#(L_{1}\cap \mathrm{Ad}(g)L_{1})=24$ for $M\cong  SO(8)/(U(1)^2 \times U(2))$ and 
$\#(L_{1}\cap \mathrm{Ad}(g)L_{1})=12$ for $M\cong  SO(8)/(U(2) \times SO(4))$. 
These are all different from the intersection numbers for the case of $L_{1}$ and $L_{2}$.

As shown in \cite{IIOST}, when $L_1$ and $\mathrm{Ad}(g)L_1$ intersect discretely for $g\in G$, 
their intersection number is equal to the sum $SB(L_1;\mathbb{Z}_2)$ of the $\mathbb{Z}_2$-betti number of $L_1$. 
Therefore, it should be noted that, by the outer automorphisms of $G$, the intersection number obtained in our case is strictly smaller than $SB(L_1;\mathbb{Z}_2)$.

\section*{Acknowledgment}
The authors would like to thank Professor Yoshihiro Ohnita for his kind explanation of the paper \cite{Ohnita} for them. 

The authors would also like to thank Professor Hiroshi Iriyeh for his valuable advice on symplectic geometry.

The first author would like to thank Professor Peter Quast for his warm hospitality during a stay at the University of Augsburg, part of which was spent preparing this manuscript.

\end{document}